\documentclass[preprint,3p,number]{elsarticle}

\usepackage{amssymb,amsfonts,amsopn,amsthm,amsmath}
\usepackage{graphicx,epstopdf}
\usepackage{tabularx,booktabs,ctable,threeparttable}
\usepackage{multirow}
\usepackage{algorithm,algorithmic}
\usepackage{caption}
\usepackage{subcaption}
\usepackage{lineno}
\usepackage{marvosym}
\usepackage{extarrows}
\usepackage{blkarray}
\usepackage[colorlinks,citecolor=blue,urlcolor=blue]{hyperref}

\newtheorem{Def}{Definition}[section]
\newtheorem{prop}{Proposition}[section]

\newtheorem{Lem}{Lemma}[section]

\numberwithin{equation}{section}
\newtheorem{theorem}{Theorem}[section]
\newtheorem{remark}{Remark}[section]
\numberwithin{Lem}{section}
\numberwithin{Def}{section}
\numberwithin{prop}{section}

\let\oldequation\equation
\let\oldendequation\endequation
\renewenvironment{equation}
{\linenomathNonumbers\oldequation}
{\oldendequation\endlinenomath}

\let\oldalign\align
\let\oldendalign\endalign
\renewenvironment{align}
{\linenomathNonumbers\oldalign}
{\oldendalign\endlinenomath}

\begin{document}
\date{}


\begin{frontmatter}


\title{Backward error analysis of the Lanczos bidiagonalization with reorthogonalization} 
\author[1,2]{Haibo Li \corref{cor1}}
\ead{haibolee1729@gmail.com}

\author[2]{Guangming Tan \fnref{fn2}}
\ead{tgm@ict.ac.cn}

\author[2]{Tong Zhao}
\ead{zhaotong@ict.ac.cn}

\affiliation[1]{organization={Computing System Optimization Lab, Huawei Technologies},
	city={Beijing},
	postcode={100094},
	country={China}}
\affiliation[2]{organization={Institute of Computing Technology, Chinese Academy of Sciences},
	city={Beijing},
	postcode={100190},
	country={China}}

\cortext[cor1]{Corresponding author}
\fntext[fn2]{This author is supported in part by the National Natural Science Foundation of China under Grant No. 62032023.}

\begin{abstract}
	The $k$-step Lanczos bidiagonalization reduces a matrix $A\in\mathbb{R}^{m\times n}$ into a bidiagonal form $B_k\in\mathbb{R}^{(k+1)\times k}$ while generates two orthonormal matrices $U_{k+1}\in\mathbb{R}^{m\times (k+1)}$ and $V_{k+1}\in\mathbb{R}^{n\times {(k+1)}}$. However, any practical implementation of the algorithm suffers from loss of orthogonality of $U_{k+1}$ and $V_{k+1}$ due to the presence of rounding errors, and several reorthogonalization strategies are proposed to maintain some level of orthogonality. In this paper, by writing various reorthogonalization strategies in a general form we make a backward error analysis of the Lanczos bidiagonalization with reorthogonalization (LBRO). Our results show that the computed $B_k$ by the $k$-step LBRO of $A$ with starting vector $b$ is the exact one generated by the $k$-step Lanczos bidiagonalization  of $A+E$ with starting vector $b+\delta_{b}$ (denoted by LB($A+E,b+\delta_{b}$)), where the 2-norm of perturbation vector/matrix $\delta_{b}$ and $E$ depend on the roundoff unit and orthogonality levels of $U_{k+1}$ and $V_{k+1}$. The results also show that the 2-norm of $U_{k+1}-\bar{U}_{k+1}$ and $V_{k+1}-\bar{V}_{k+1}$ are controlled by the orthogonality levels of $U_{k+1}$ and $V_{k+1}$, respectively, where $\bar{U}_{k+1}$ and $\bar{V}_{k+1}$ are the two orthonormal matrices generated by the $k$-step LB($A+E,b+\delta_{b}$) in exact arithmetic. Thus the $k$-step LBRO is mixed forward-backward stable as long as the orthogonality of $U_{k+1}$ and $V_{k+1}$ are good enough. We use this result to investigate the backward stability of LBRO based SVD computation algorithm and LSQR algorithm. Numerical experiments are made to confirm our results.
\end{abstract}


\end{frontmatter}

\section{Introduction}\label{sec1}
In \cite{Golub1965}, Golub and Kahan propose an algorithm for reducing an arbitrary rectangle matrix to upper bidiagonal form, which is the first step for computing the singular value decomposition(SVD) of the given matrix. Later in \cite{Paige1982}, Paige and Saunders propose an algorithm which reduces a matrix to lower bidiagonal form iteratively, based on which they propose the most widely used LSQR algorithm for solving large sparse least squares problems. Although both the upper and lower bidiagonal reductions can be directly achieved by Householder transformations \cite{Golub2013}, for large sparse or structured matrices, one often need to reduce the matrix to a partial upper or lower bidiagonal form iteratively, which can be obtained by using Lanczos-type procedures. Such two procedures are proposed in \cite{Paige1982}, which are called BIDIAG-1 and BIDIAG-2 corresponding to the lower and upper bidiagonal reductions, respectively. For these reasons, the two above procedures are usually called Golub-Kahan-Lanczos bidiagonalizations or simply Lanczos bidiagonalizations, and we focus on the lower Lanczos bidiagonalization in this paper.

Given a matrix $A \in \mathbb{R}^{m\times n}$ with $m\geq n$. First choose a nonzero $b\in\mathbb{R}^{m}$ as the starting vector and compute
\[ \beta_{1}u_1 = b, \ \ \ \alpha_{1}v_{1}=A^{T}u_{1}\]
such that $\|u_{1}\|=\|v_{1}\|=1$ and $\beta_{1},\alpha_{1}>0$.  Throughout the rest of the paper $\|\cdot\|$ always denotes either the vector or matrix 2-norm. Then for $i=1,2,\dots$, the Lanczos bidiagonalization computes vectors $u_{i+1}$ and $v_{i+1}$ of unit length using the recurrences
\begin{align*}
	& \beta_{i+1}u_{i+1} = Av_{i} - \alpha_{i}u_{i} ,  \\
	&  \alpha_{i+1}v_{i+1} = A^{T}u_{i+1} - \beta_{i+1}v_{i} , 
\end{align*}
where $\alpha_{i+1}, \beta_{i+1}>0$. If $\alpha_{i+1}$ or $\beta_{i+1}$ is zero, then the procedure terminates, having found an invariant singular subspace of $A$, and this is usually called ``lucky terminate" \cite{Golub1965}. In the paper we suppose that the procedure does not terminate after $k$ steps. In exact arithmetic the Lanczos bidiagonalization computes two group of mutual orthogonal vectors $\{u_{1}, \dots, u_{k+1}\}$ and $\{v_{1} \dots, v_{k+1}\}$ which are called Lanczos vectors, and a lower bidiagonal matrix
\[ B_k = \begin{pmatrix}
		\alpha_{1} & & & \\
		\beta_{2} &\alpha_{2} & & \\
		&\beta_{3} &\ddots & \\
		& &\ddots &\alpha_{k} \\
		& & &\beta_{k+1}
	\end{pmatrix} \in \mathbb{R}^{(k+1)\times k} .\]
It can be proved that $\{u_{1},\dots,u_{k+1}\}$ and $\{v_{1},\dots,v_{k+1}\}$ are orthonormal bases of Krylov subspaces $\mathcal{K}_{k+1}(AA^{T},b)$ and $\mathcal{K}_{k+1}(A^{T}A,A^{T}b)$, respectively \cite{Larsen1998}. The fundamental relations of the $k$-step Lanczos bidiagonalization can be written in the matrix form
\begin{align*}
	& U_{k+1}(\beta_{1}e_{1}^{(k+1)}) = b , \\
	& AV_{k} = U_{k+1}B_{k} , \\
	& A^{T}U_{k+1} = V_{k}B_{k}^{T}+\alpha_{k+1}v_{k+1}(e_{k+1}^{(k+1)})^{T} .
\end{align*}
where $e_{i}^{(k+1)}$ is the $i$-th column of the identity matrix of order $k+1$, and $U_{k+1} = (u_{1}, \dots, u_{k+1})$ and $V_{k+1}=(v_{1},\dots, v_{k+1})$ are two orthonormal matrices. 

The Lanczos bidiagonalization is widely used for designing efficient algorithms for many type of large sparse or structured matrix problems. Note that $B_{k}$ is the Ritz-Galerkin projection of $A$ on the subspaces $\mathcal{K}_{k+1}(AA^{T},b)$ and $\mathcal{K}_{k}(A^{T}A,A^{T}b)$. Therefore, the extreme singular values and corresponding vectors of $A$ can be well approximated by the SVD of $B_{k}$ \cite{Golub1965}. It is shown in \cite{Simon2000} that good low-rank approximations of matrices can be directly obtained from the Lanczos bidiagonalization applied to the given matrix without computing any SVD. For large sparse least squares problems of the form $\min_{x\in\mathbb{R}^{n}}\|Ax-b\|$, the most commonly used LSQR solver is also based on the Lanczos bidiagonalization. Furthermore, for discrete linear ill-posed inverse problems, several Lanczos bidiagonalization based iterative regularization algorithms are developed for solving large sparse problems, where a regularized solution is obtained by a proper early stopping criterion or by a hybrid regularization method; see \cite{Bjorck1988,Hansen1998,Hansen2010} for further discussions.

In finite precision arithmetic, due to the influence of rounding errors, the Lanczos vectors computed by the Lanczos bidiagonalization gradually lose their mutual orthogonality as the iteration progresses \cite{Golub1965,Larsen1998}. This is a typical phenomenon that appears in the Lanczos-type algorithms, which is first observed in the symmetric Lanczos process used for computing some extreme eigenvalues and eigenvectors of a symmetric matrix \cite{Lanczos1950}. The loss of orthogonality of Lanczos vectors will lead to a delay of convergence in the computation of eigenvalues and eigenvectors, and sometimes it is also difficult to determine whether some computed approximations are additional copies or genuine close eigenvalues \cite{Paige1971, Paige1976, Paige1980}. The above properties can be adapted to handle the Lanczos bidiagonalization since the Lanczos bidiagonalization of $A$ with staring vector $b$ can be written as the symmetric Lanczos process of $\begin{pmatrix} O & A \\ A^{T} & O \end{pmatrix}$ with staring vector $\begin{pmatrix} b \\ 0 \end{pmatrix}$ \cite{Bjorck1996}. On the other hand, when using the LSQR to solve least squares problems, the loss of orthogonality may cause the algorithm requiring much more iterations to converge; the finite precision behavior of LSQR is very similar to the closely related conjugate gradient (CG) algorithm based on symmetric Lanczos process, and we refer to \cite{Green1992,Meurant2006,Green2021}. For discrete linear inverse problems, the Lanczos bidiagonalization based regularization algorithms also suffers from the delay of convergence of regularized solutions, which can make the propagation of noise during iterations rather irregular \cite{Hnetyn2009}. For these reasons, the Lanczos bidiagonalization is usually performed with reorthogonalization for solving least squares problems and discrete linear inverse problems. There are several reorthogonalization strategies proposed to maintain some level of orthogonality, such as partial reorthogonalization \cite{Larsen1998} and one-sided reorthogonalization \cite{Simon2000}. 

It is well known that algorithms based on a sequence of Householder transformations have very good stability properties \cite{Higham2002}, and these properties have been used to show that the Householder (upper) bidiagonal reduction is mixed forward-backward stable \cite[Theorem A2]{Byers2008}. For Lanczos bidiagonalization with a reorthogonalization strategy, however, very little is known about the numerical stability of it. As far as we know, the only one result is about the one-sided reorthogonalization, which states that the process applied to a matrix $C$ in finite precision arithmetic produces Krylov subspaces generated by a nearby matrix $C+E_1$, where $E_1$ is an error matrix \cite{Barlow2013}. In this paper, we write those various types of reorthogonalization strategies in a general form that at each iteration $u_k$ and $v_k$ are reorthogonalized against some previous vectors among $\{u_1,\dots, u_{k-1}\}$ and $\{v_1,\dots, v_{k-1}\}$, respectively. Note that some vectors may not be included, which means that they are not used in the reorthogonalization step. Using this form, we can analyze numerical stability of the Lanczos bidiagonalization with reorthogonalization (LBRO), which is the main purpose of this paper. 

In this paper, we give a backward error analysis of the LBRO and generalize the result of \cite{Barlow2013}. Denote the roundoff unit by $\mathbf{u}$. We first establish a relationship between the LBRO and Householder transformation based bidiagonal reduction. Based on this result, we show that for the $k$-step LBRO of $A$ with starting vector $b$ (denoted by LBRO($A,b$)): (1). the computed $B_k$ is the exact one generated by the $k$-step LB($A+E,b+\delta_{b}$), where the perturbation quantity $\|\delta_{b}\|/\|b\|=O(\mathbf{u})$ and $\|E\|/\|A\|$ is controlled by $\mathbf{u}$ and orthogonality levels of $U_{k+1}$ and $V_{k+1}$; (2). if we denote the two orthonormal matrices generated by LB($A+E,b+\delta_{b}$) in exact arithmetic by $\bar{U}_{k+1}$ and $\bar{V}_{k+1}$, respectively, then $\|U_{k+1}-\bar{U}_{k+1}\|$ and $\|V_{k+1}-\bar{V}_{k+1}\|$ are controlled by the orthogonality levels of $U_{k+1}$ and $V_{k+1}$, respectively. Compared with \cite[Theorem 5.2]{Barlow2013} that can only deal with the $n$-step procedure with one-sided reorthogonalization, our result can apply to the $k$-step LBRO for $1\leq k \leq n$. Following Higham \cite[\S 1.5]{Higham2002}, our result implies that the $k$-step LBRO is mixed forward-backward stable as long as the orthogonality of $U_{k+1}$ and $V_{k+1}$ are good enough. We then use this result to investigate backward stability of LBRO based algorithms including SVD computation and LSQR.

The paper is organized as follows. In Section \ref{sec2}, we review reorthogonalization strategies for  the Lanczos bidiagonalization and give some properties.  In Section \ref{sec3}, we first establish a relationship between the LBRO and Householder transformation based bidiagonal reduction, then we derive mixed backward-forward error bounds for the LBRO. In Section \ref{sec4}, our result is applied to discuss backward stability of LBRO based SVD computation algorithm and LSQR. In Section \ref{sec5}, we use some numerical examples to illustrate the results. Finally, we conclude the paper in Section \ref{sec6}.

Throughout the paper, we denote by $I_{k}$ and $O_{k\times l}$ the identity matrix of order $k$ and zero matrix of order $k\times l$, respectively, by $e_{i}^{(k)}$ the $i$-th column of $I_k$ and by $0_{l}$ the zero vector of dimension $l$.

\section{The Lanczos bidiagonalization and reorthogonalization strategies}\label{sec2}
In this section, we review the Lanczos bidiagonalization in finite precision arithmetic and reorthogonalization strategies. From now on, quantities $\alpha_{i}$, $u_i$, $B_k$, etc. denote the computed ones in finite precision arithmetic. Several types of reorthogonalization strategies have been proposed for maintaining some level of orthogonality of Lanczos vectors \cite{Larsen1998,Simon2000,Barlow2013}, all of which can be written in the following form.

Suppose that at the $i$-th step, the bidiagonalization procedure have computed 
\begin{equation}\label{2.1}
	\beta_{i+1}^{'}u_{i+1}^{'}=Av_{i}-\alpha_{i}u_{i}-f_{i}^{'},
\end{equation}
where $u_{i+1}^{'}$ and $\beta_{i+1}^{'}$ are two temporary quantities and $f_{i}^{'}$ is a rounding error term. A reorthogonalization strategy applied to $u_{i+1}$ means that we choose $i$ real numbers $\xi_{1i}, \dots, \xi_{ii}$ and form
\begin{equation}\label{2.2}
	\beta_{i+1}u_{i+1} =\beta_{i+1}^{'}u_{i+1}^{'}-\sum\limits_{j=1}^{i}\xi_{ji}u_{j}-f_{i}^{''} ,
\end{equation}
where $f_{i}^{''}$ is a rounding error term. Then the algorithm will be continued with $u_{i+1}$ and $\beta_{i+1}$ instead of $u_{i+1}^{'}$ and $\beta_{i+1}^{'}$. The reorthogonalization step of $u_{i+1}$ aims to maintain some level of orthogonality among $u_{i+1}$ and $u_{j}$ for $j=1,\dots,i$. In \eqref{2.2}, the choice of coefficients $\xi_{1i}, \dots, \xi_{ii}$ varies from different types of reorthogonalization strategies, and some values of the coefficients may be zero, which means that the corresponding Lanczos vectors are not used in the reorthogonalization step. 

Combining \eqref{2.1} and \eqref{2.2}, the reorthogonalization step of $u_{i+1}$ can be written as the recurrence
\begin{equation}\label{2.3}
	Av_{i}=\alpha_{i}u_{i}+\beta_{i+1}u_{i+1}+ \sum\limits_{j=1}^{i}\xi_{ji}u_{j}+f_{i} ,
\end{equation}
where $f_{i}=f_{i}^{'}+f_{i}^{''}$ is the rounding error term. The reorthogonalization step of $v_{i}$ is similar to that of $u_{i+1}$ and can be written as the recurrence
\begin{equation}\label{2.4}
	A^{T}u_{i}=\alpha_{i}v_{i}+\beta_{i}v_{i-1}+
	\sum\limits_{j=1}^{i-1}\eta_{ji}v_{i}+g_{i} ,
\end{equation}
where $g_{i}$ is the rounding error term. Note that for $i=1$ reorthogonalization of $v_{1}$ is not needed and $v_{0}$ is a zero vector.

\begin{remark}\label{remark2.1}
	For matrix $A \in \mathbb{R}^{m\times n}$ with $m=n$, the Lanczos bidiagonalization must terminate at $k=n$ in exact arithmetic. In finite precision arithmetic, however, $\beta_{n+1}$ is usually nonzero. In this case, the computation of $u_{n+1}$ does not make any sense since $U_{n+1}=(U_{n},u_{n+1})\in \mathbb{R}^{n\times (n+1)}$ has deficient column rank, and thus the LBRO should not reorthogonalize $u_{n+1}$ if $m=n$. In practice, the Lanczos bidiagonalization is usually performed in $k\ll n$ steps. 
\end{remark}

Now we state a set of properties and assumptions on the finite precision behaviors of the Lanczos bidiagonalization, which are from the results of rigorous analysis of the symmetric Lanczos process and Lanczos bidiagonalization \cite{Paige1976, Parlett1998, Simon1984a, Larsen1998, Simon2000}. They constitute a model for the actual computations and include essential features while discard irrelevant ones. First, by \eqref{2.3} and \eqref{2.4} the Lanczos bidiagonalization with reorthogonalization can be written in matrix form
\begin{align}
	& AV_{k}=U_{k+1}(B_{k}+C_{k})+F_{k}  , \label{2.5} \\
	& A^{T}U_{k+1}=V_{k}(B_{k}^{T}+D_{k})+\alpha_{k+1}v_{k+1}(e_{k+1}^{(k+1)})^{T}+G_{k+1}  , \label{2.6}
\end{align}
where
\[ C_{k}=\begin{pmatrix}
		\xi_{11} &\xi_{12} &\dots &\xi_{1k} \\
		0& \xi_{22}&\cdots &\xi_{2k} \\
		& 0&\ddots &\vdots \\
		& &\ddots  &\xi_{kk} \\
		& & &0
	\end{pmatrix} \in \mathbb{R}^{(k+1)\times k}, \ \
	D_{k}=\begin{pmatrix}
		0 &\eta_{12} &\eta_{13} &\cdots &\eta_{1k} &\eta_{1k+1} \\
		&0 &\eta_{23} &\eta_{24} &\cdots &\eta_{2k+1} \\
		& &\ddots &\ddots &\ddots &\vdots \\
		& & &\ddots &\eta_{k-1,k} &\eta_{k-1,k+1} \\
		& & & &0 & \eta_{k,k+1}
	\end{pmatrix} \in \mathbb{R}^{k \times (k+1)},\]
and $F_{k}=(f_{1}, \dots,f_{k})$ and $G_{k+1}=(g_{1},\dots,g_{k+1})$, satisfying $\lVert F_k\lVert, \lVert G_{k+1}\lVert = O(\lVert A\lVert \mathbf{u})$ \cite{Simon1984a,Parlett1998}. Note that the rounding error terms in \eqref{2.1} and \eqref{2.2} are also satisfied $\|f_{i}^{'}\|, \|f_{i}^{''}\|=O(\lVert A\lVert \mathbf{u})$ \cite{Simon1984a}. Then, we always assume that the computed Lanczos vectors are of unit length in order to simplify rounding error analysis. Finally, we assume that
\begin{equation*}\label{2.11}
	\mbox{no} \ \alpha_{i} \ \mbox{and} \ \beta_{i} \ \mbox{ever become negligible},
\end{equation*}
which is almost always true in practice, and the rare cases where $\alpha_{i}$ or $\beta_{i}$ do become very small are actually the lucky ones, since then the algorithm should be terminated, having found an invariant singular subspace \cite{Simon1984a}. 

Following \cite{Barlow2013,Paige2009}, we define the orthogonality level of Lanczos vectors as follows. 
\begin{Def}
	Let $\mathbf{SUT(\cdot)}$ denotes the strictly upper triangular part of a matrix. The orthogonality level of Lanczos vectors $\{u_{1},\dots,u_{k}\}$ and the corresponding matrix $U_{k}=(u_{1},\dots,u_{k})$ is 
	\begin{align*}
		\mu_{k} =  \lVert \mathbf{SUT}(I_{k}-U_{k}^{T}U_{k})\lVert, 
	\end{align*}
	while the orthogonality level of $\{v_{1},\dots,v_{k}\}$ and  $V_{k}=(v_{1},\dots,v_{k})$ is 
	\begin{align*}
		\nu_{k} =  \lVert \mathbf{SUT}(I_{k}-V_{k}^{T}V_{k})\lVert.
	\end{align*}
\end{Def}

By the Cauchy's interlacing theorem for singular values, it can be verified that $\mu_{i}\leq\mu_{i+1}$ and $\nu_{i}\leq\nu_{i+1}$. Let $\sigma_{i}(\cdot)$ denote the $i$-th largest singular value of a matrix. It is shown in \cite{Barlow2013} that 
\begin{equation}\label{2.16}
	\sigma_{1}(V_{k}) \leq 1 + \nu_{k}, 
\end{equation}
and
\begin{equation}\label{2.17}
	\sigma_{k}(V_{k}) \geq (1-2 \nu_{k})^{1/2} = 1 - \nu_{k} + O(\nu_{k}^{2})
\end{equation}
if $\nu_{k} < 1/2$.
Similarly, for $U_{k}$ we have
\begin{equation}\label{2.14}
	\sigma_{1}(U_{k}) \leq 1 + \mu_{k},
\end{equation}
and 
\begin{equation}\label{2.15}
	\sigma_{k}(U_{k}) \geq (1-2 \mu_{k})^{1/2} = 1 - \mu_{k} + O(\mu_{k}^{2})
\end{equation}
if $\mu_{k} < 1/2$. In the rest of the paper, we always assume $\mu_{i}<1/2$ and $\nu_{i}<1/2$.

Rewrite \eqref{2.3} as
$$Av_{k}=U_{k+1}\tilde{c}_{k}+f_{k},$$
where $\tilde{c}_{k}=(\xi_{1k},\dots,\xi_{k-1,k},\alpha_{k}+\xi_{kk},\beta_{k+1})^{T}$. Then using \eqref{2.15} we obtain
\begin{align}\label{2.20}
	\begin{split}
		\lVert \tilde{c}_{k} \lVert
		&= \lVert U_{k+1}^{\dag}(Av_{k}-f_{k}) \lVert
		\leq \sigma_{k+1}(U_{k+1})^{-1}(\lVert A \lVert + \lVert f_{k} \lVert) \\
		&\leq [1+\mu_{k+1}+O(\mu_{k+1}^{2})][\lVert A \lVert+O(\lVert A \lVert\mathbf{u})] \\
		&= \lVert A \lVert  + O(\lVert A \lVert(\mathbf{u}+\mu_{k+1})),
	\end{split}
\end{align}
where we neglect high order terms of $\mathbf{u}$ and $\mu_{k+1}$. Thus from \eqref{2.20} we have 
\begin{equation}\label{bnd_beta}
	\beta_{k+1}\leq\lVert A \lVert  + O(\lVert A \lVert(\mathbf{u}+\mu_{k+1})).
\end{equation}
Similarly, if we rewrite \eqref{2.4} as 
$$A^{T}u_{k}=V_{k}\tilde{d}_{k}+g_{k}$$
where $\tilde{d}_{k}=(\eta_{1k},\dots,\eta_{k-2,k},\beta_{k}+\eta_{k-1k},\alpha_{k})^T$, then we can obtain
\begin{equation}\label{2.21}
	\lVert \tilde{d}_{k} \lVert \leq 
	\lVert A \lVert  + O(\lVert A \lVert(\mathbf{u}+\nu_{k}))
\end{equation}
and thus 
\begin{equation}\label{bnd_alpha}
	\alpha_{k}\leq\lVert A \lVert + O(\lVert A \lVert(\mathbf{u}+\nu_{k})).
\end{equation}

The next result is about an upper bound on the coefficients of reorthogonalization step of $u_{k+1}$, which will be used later in the backward error analysis.
\begin{prop}\label{prop2.1}
	Let $\bar{c}_{k}=(\xi_{1k},\dots,\xi_{kk})^{T}$. Then 
	\begin{equation}\label{2.19}
		\lVert \bar{c}_{k} \lVert = O(\lVert A \lVert(\mathbf{u}+\mu_{k+1}+\nu_{k})).
	\end{equation}
\end{prop}
\begin{proof}
	At the $k$-th step, recurrence \eqref{2.3} can be written as
	\begin{equation}\label{eq1}
		U_{k}\bar{c}_{k} = Av_{k}-\alpha_{k}u_{k}-\beta_{k+1}u_{k+1}-f_{k}.
	\end{equation}
	Therefore, we get
	\begin{align*}
		U_{k}^{T}U_{k}\bar{c}_{k}=
		U_{k}^{T}Av_{k}-\alpha_{k}U_{k}^{T}u_{k}-\beta_{k+1}U_{k}^{T}u_{k+1}-U_{k}^{T}f_{k},
	\end{align*}
	which can also be written as
	\begin{align}\label{2.18}
		U_{k}^{T}U_{k}\bar{c}_{k}=(B_{k-1}^{T}+D_{k-1})^{T}V_{k-1}^{T}v_{k}-\alpha_{k}(U_{k}^{T}u_{k}-e_{k}^{(k)})-\beta_{k+1}U_{k}^{T}u_{k+1}+G_{k}^{T}v_{k}-U_{k}^{T}f_{k},
	\end{align}
	where we use the relation
	$U_{k}^{T}Av_{k}
	=[V_{k-1}(B_{k-1}^{T}+D_{k-1})+\alpha_{k}v_{k}(e_{k}^{(k)})^{T}+G_{k}]^{T}v_{k}$ derived from \eqref{2.6}.
	
	Now we give an upper bound on $\|B_{k-1}^{T}+D_{k-1}\|$. Using \eqref{2.6} we have
	$$V_{k-1}^{T}V_{k-1}(B_{k-1}^{T}+D_{k-1})=V_{k-1}^{T}A^TU_{k}-\alpha_{k}V_{k-1}^{T}v_{k}(e_{k}^{(k)})^{T}-V_{k-1}^{T}G_{k},$$
	which leads to
	\begin{align}\label{bound1}
		\begin{split}
			\|B_{k-1}^{T}+D_{k-1}\| &\leq \|(V_{k-1}^{T}V_{k-1})^{-1}\|\cdot \|V_{k-1}^{T}A^TU_{k}-\alpha_{k}V_{k-1}^{T}v_{k}e_{k}^{T}-V_{k-1}^{T}G_{k}\| \\
			&\leq (1-2\nu_{k-1})^{-1} [\|A\|(1+\nu_{k-1})(1+\mu_{k})+\alpha_{k}\nu_{k}+(1+\nu_{k-1})\|G_{k}\|] \\
			&= [1+2\nu_{k-1}+O(\nu_{k-1}^{2})]\cdot[\|A\|(1+\nu_{k-1}+\mu_{k}+\nu_{k-1}\mu_{k})+ \\
			& \ \ \ \ \ (\lVert A \lVert + O(\lVert A \lVert(\mathbf{u}+\nu_{k}))\nu_{k}+(1+\nu_{k-1})O(\|A\|\mathbf{u})] \\
			&= \|A\|+O(\|A\|(\mathbf{u}+\nu_{k}+\mu_{k})),
		\end{split}
	\end{align}
	by neglecting high order terms, where we use the inequality
	$$\lVert(V_{k-1}^{T}V_{k-1})^{-1}\lVert=(\sigma_{k-1}(V_{k-1}))^{-2}\leq(1-2\nu_{k-1})^{-1} = 1+2\nu_{k-1}+O(\nu_{k-1}^{2})$$ 
	derived from \eqref{2.17}.
	
	By using upper bounds on $\alpha_k$ and $\beta_{k+1}$ in \eqref{bnd_alpha} and \eqref{bnd_beta}, we get
	\begin{align*}
		\begin{split}
			& \ \ \ \ \lVert -\alpha_{k}(U_{k}^{T}u_{k}-e_{k}^{(k)})-\beta_{k+1}U_{k}^{T}u_{k+1}+G_{k}^{T}v_{k}-U_{k}^{T}f_{k} \lVert \\
			&\leq \alpha_{k}\mu_{k} +\beta_{k+1}\mu_{k+1} + (1+1+\mu_{k})O(\lVert A \lVert\mathbf{u}) \\
			&= O(\lVert A\lVert(\mathbf{u}+\mu_{k+1})) .
		\end{split}
	\end{align*}
	Using the inequality derived from \eqref{2.15}
	$$\lVert(U_{k}^{T}U_{k})^{-1}\lVert = (\sigma_{k}(U_{k}))^{-2} \leq (1-2\mu_{k})^{-1} = 1+2\mu_{k}
	+ O(\mu_{k}^{2}),$$
	and by neglecting high order terms, we finally obtain from \eqref{2.18} that
	\begin{align*}
		\|\bar{c}_{k}\| &\leq \|(U_{k}^{T}U_{k})^{-1}\|\cdot[\|B_{k-1}^{T}+D_{k-1}\|\|V_{k-1}^{T}v_{k}\| +O(\lVert A\lVert(\mathbf{u}+\mu_{k+1}))] \\
		&\leq [1+2\mu_{k}
		+ O(\mu_{k}^{2})]\cdot[(\|A\|+O(\|A\|(\mathbf{u}+\nu_{k}+\mu_{k-1})))\nu_{k}+O(\lVert A\lVert(\mathbf{u}+\mu_{k+1}))] \\
		&= O(\lVert A \lVert(\mathbf{u}+\mu_{k+1}+\nu_{k})),
	\end{align*}
	which is the desired result.
\end{proof}

At the end of this section, we briefly investigate the orthogonality level between two contiguous Lanczos vectors. Is is shown from \cite{Paige1976} that the property of local orthogonality holds for Lanczos vectors of symmetric Lanczos process. In fact, this property also applies to the Lanczos bidiagonalization (without reorthogonalization; see \eqref{2.1}), which can be written in the following form:
\begin{equation}\label{local_orth}
	\beta_{i+1}^{'}|u_{i}^{T}u_{i+1}^{'}| = O(c_{1}(m,n)\|A\|\mathbf{u}),
\end{equation}
where $c_{1}(m,n)$ is a moderate constant depending on $m$ and $n$ \cite{Paige1976,Simon1984a}. The property of local orthogonality for $v_{i}$ is similar and we omit it. For this reason, in some literature it is proposed that $u_{i}$ is not needed in the reorthogonalization step of $u_{i+1}$; see e.g. \cite{Simon1984a,Simon1984b}. This case corresponds to choose $\xi_{ii}=0$ in \eqref{2.2}. In fact, the orthogonality among $u_{i}$ and $u_{i+1}$ will not be bad as long as the orthogonality between $u_{i+1}$ and $\{u_{1},\dots,u_{i-1}\}$ is in a desired level. If we use 
\begin{equation}
	\omega_{i+1}=\max_{1\leq j \leq i-1}|u_{j}^{T}u_{i+1}|	
\end{equation}
to measure the orthogonality level between $u_{i+1}$ and $\{u_{1},\dots,u_{i-1}\}$, then we have the following result.
\begin{prop}
	Suppose that $u_{i}$ is not used in the reorthogonalization step of $u_{i+1}$. If $\mu_{i}\ll i^{-1}$, Then we have
	\begin{equation}\label{bound2}
		\beta_{i+1}|u_{i}^{T}u_{i+1}| = O(c_{1}(m,n)\lVert A \lVert\mathbf{u}) + O(\lVert A \lVert(\nu_{i}\mu_{i}+\mu_{i}^{2}+\omega_{i+1}\mu_{i})).
	\end{equation}
\end{prop}
\begin{proof}
	Since $\xi_{ii}=0$, from \eqref{2.2} we obtain
	$$\beta_{i+1}u_{i}^{T}u_{i+1}=\beta_{i+1}^{'}u_{i}^{T}u_{i+1}^{'}-
	\sum\limits_{j=1}^{i-1}\xi_{ji}u_{i}^{T}u_{j}-u_{i}^{T}f_{i}^{''}.$$
	Using \eqref{local_orth}, it follows that
	\begin{equation}\label{2.22}
		\beta_{i+1}|u_{i}^{T}u_{i+1}|\leq O(c_{1}(m,n)\|A\|\mathbf{u})+\max_{1\leq j \leq i-1}|\xi_{ji}|\cdot i\mu_{i}+O(\|A\|\mathbf{u}).	
	\end{equation}
	In order to get the desired result we need to find an upper bound on $M=\max_{1\leq j \leq i-1}|\xi_{ji}|$. Again from \eqref{2.2}, after some rearranging we obtain
	$$\xi_{lj}=\beta_{i+1}^{'}u_{l}^{T}u_{i+1}^{'}-\beta_{i+1}u_{l}^{T}u_{i+1}-\sum\limits_{j=1, j\neq l}^{i-1}(u_{l}^{T}u_{j})\xi_{ji}-u_{l}^{T}f_{i}^{''}$$
	for $l=1,\dots,i-1$. Premultiplying \eqref{2.1} by $U_{i}^{T}$, we have
	\begin{align*}
		\beta_{i+1}^{'}U_{i}^{T}u_{i+1}^{'}
		&= U_{i}^{T}Av_{i}-\alpha_{i}U_{i}^{T}u_{i}-U_{i}^{T}f_{i}^{'} \\
		&= (B_{i-1}+D_{i-1}^{T})V_{i-1}^{T}v_{i}+e_{i}^{(i)}\alpha_{i}v_{i}^{T}v_{i}+G_{i}^{T}v_{i}
		- \alpha_{i}U_{i}^{T}u_{i}-U_{i}^{T}f_{i}^{'} \\
		&= (B_{i-1}+D_{i-1}^{T})V_{i-1}^{T}v_{i} - \alpha_{i}(U_{i}^{T}u_{i}-e_{i}^{(i)}) +G_{i}^{T}v_{i}-U_{i}^{T}f_{i}^{'}.
	\end{align*}
	By using \eqref{bound1}, after some calculations, we have
	$$\beta_{i+1}^{'}|u_{l}^{T}u_{i+1}^{'}| \leq \beta_{i+1}^{'}\|U_{i}^{T}u_{i+1}^{'}\|=
	O(\|A\|(\mathbf{u}+\mu_{i}+\nu_{i})),$$
	and thus
	\begin{align*}
		|\xi_{li}| \leq O(\|A\|(\mathbf{u}+\mu_{i}+\nu_{i}))+\beta_{i+1}\max_{1\leq j \leq i-1}|u_{j}^{T}u_{i+1}|+i\mu_{i-1}M+O(\|A\|\mathbf{u}).
	\end{align*}
	Now the right-hand side does not depend on $l$ anymore, and we obtain by taking the maximum on the left side that
	\begin{align*}
		(1-i\mu_{i-1})M &\leq O(\|A\|(\mathbf{u}+\mu_{i}+\nu_{i}))+\beta_{i+1}\max_{1\leq j \leq i-1}|u_{j}^{T}u_{i+1}|+O(\|A\|\mathbf{u}) \\
		&= O(\|A\|(\mathbf{u}+\mu_{i}+\nu_{i}+\omega_{i+1}))
	\end{align*}
	Since $\mu_{i-1}\ll i^{-1}$ we obtain
	$$M=O(\|A\|(\mathbf{u}+\mu_{i}+\nu_{i}+\omega_{i+1})).$$
	By \eqref{2.22} we finally obtain the desired bound.
\end{proof}

For semiorthogonalization strategy and partial reorthogonalization \cite{Simon1984a,Simon1984b,Larsen1998}, the orthogonality levels of $U_{i}$ and $V_{i}$ are kept below $O(\sqrt{\mathbf{u}})$, thus at the $k$-th step we have $\nu_{k},\mu_{k}=O(\sqrt{\mathbf{u}})$. It follows from \eqref{bound2} that $\beta_{i+1}|u_{i}^{T}u_{i+1}| = O(c_{1}(m,n)\lVert A \lVert\mathbf{u})$ as long as we keep $\omega_{k+1}=O(\sqrt{\mathbf{u}})$. Thus the property of local orthogonality still holds for the semiorthogonalization strategy.

\section{Backward error analysis of the LBRO}\label{sec3}
In this section, we first establish a relationship between the LBRO and Householder transformation based bidiagonal reduction. Then we give a backward error analysis of the LBRO to show the mixed forward-backward stability property.

\subsection{Connections with Householder bidiagonal reduction}

Before giving the results in finite precision arithmetic, we first show a connection between the Lanczos bidiagonalization and Householder QR factorization in exact arithmetic. Is is shown in \cite{Bjorck1992} that the modified Gram-Schmidt(MGS) procedure for the QR factorization of a matrix $C\in \mathbb{R}^{r\times l}$ with $r\geq l$ can be interpreted as the Householder QR factorization applied to the augmented matrix $\bar{C}=\begin{pmatrix}
	O_{l\times l} \\ C
\end{pmatrix}$, which is not only true in exact arithmetic, but also in the presence of rounding errors
as well. To see this equivalence, let $q_{1},\dots,q_{l}\in \mathbb{R}^{l}$ be vectors obtained by applying the MGS procedure to $C$, and the corresponding compact QR factorization of $C$ is $C=QR$ where $Q=(q_{1},\dots,q_{l})$. Then the Householder QR factorization of $\bar{C}$ is
\begin{equation} \label{3.10}
	(W_{l}\cdots W_{1})\bar{C} =
	\begin{pmatrix}
		R \\ O_{r\times l}
	\end{pmatrix} ,
\end{equation}
where $W_{j}$ are Householder matrices:
\begin{equation*}
	W_{j}=I_{r+l}-w_{j}w_{j}^{T} , \ \
	w_{j} = \begin{pmatrix}
		-e_{j}^{(l)} \\
		q_{j}
	\end{pmatrix} \in \mathbb{R}^{r+l} .
\end{equation*}
Thus the $k$-th step Householder transformation $W_{k}\cdots W_{1}\bar{C}$ is identical to applying $k$ steps MGS procedure to $C$ \cite{Bjorck1992}.

For the LBRO a similar property holds. We prove the result for $m>n$, while we discuss the case of $m=n$ in the remark. In order to avoid notation confusions, for all quantities computed in exact arithmetic, we add `` $\hat{}$ " to $\alpha_{i}$, $u_{i}$, $B_{k}$, etc. to denote the corresponding quantities. 
\begin{prop}\label{prop3.1}
	For the $k$-step Lanczos bidiagonalization in exact arithmetic, let 
	\begin{equation*}
		\hat{P}_{i}=I_{m+n+1}-\hat{p}_{i}\hat{p}_{i}^{T}, \ \
		\hat{p}_{i}=\begin{pmatrix}
			-e_{i}^{(n+1)} \\
			\hat{u}_{i}
		\end{pmatrix} \in \mathbb{R}^{m+n+1}.
	\end{equation*}
	Then we have	
	\begin{equation} \label{3.13}
		\begin{pmatrix}
			O_{(n+1)\times k} \\
			A\hat{V}_{k}
		\end{pmatrix} =
		\hat{P}_{1} \cdots \hat{P}_{k+1}\begin{pmatrix}
			\hat{B}_{k} \\
			O_{s\times k}
		\end{pmatrix}, \ \ s=m+n-k.
	\end{equation}
\end{prop}	
\begin{proof}
	After performing the procedure in exact arithmetic $n$ steps (if the procedure terminates at some step, we can choose a new starting vector and continue on), we have
	\begin{equation*}
		\begin{pmatrix}
			b, & A\hat{V}_{n}
		\end{pmatrix} = \hat{U}_{n+1}
		\begin{pmatrix}
			\hat{\beta}_{1}e_{1}^{(n+1)}, & \hat{B}_{n}
		\end{pmatrix} = \hat{U}_{n+1}
		\begin{pmatrix}
			\hat{\beta}_{1} & \hat{\alpha}_{1} &  & \\
			& \hat{\beta}_{2} & \ddots &  \\
			&  & \ddots & \hat{\alpha}_{n} \\
			&  &  & \hat{\beta}_{n+1}
		\end{pmatrix},
	\end{equation*}
	which is the QR factorization of $(b, \ AV_{n})\in\mathbb{R}^{m\times(n+1)}$. Therefore, by \eqref{3.10} we have
	\begin{equation*}
		\begin{pmatrix}
			0_{n+1} &	O_{(n+1)\times n} \\
			b & A\hat{V}_{n}
		\end{pmatrix} =
		\hat{P}_{1} \cdots \hat{P}_{n+1}\begin{pmatrix}
			\hat{\beta}_{1}e_{1}^{(n+1)} & \hat{B}_{n} \\
			0_{m} & O_{m\times n}
		\end{pmatrix}.
	\end{equation*}
	Equating from the second to $(k+1)$-th column of the above matrix yields
	\begin{equation*}
		\begin{pmatrix}
			O_{(n+1)\times k} \\
			AV_{k}
		\end{pmatrix} =
		(\hat{P}_{1} \cdots \hat{P}_{k+1}) (\hat{P}_{k+2} \cdots \hat{P}_{n+1})\begin{pmatrix}
			B_{k} \\
			O_{s\times k}
		\end{pmatrix} =
		\hat{P}_{1} \cdots \hat{P}_{k+1}\begin{pmatrix}
			B_{k} \\
			O_{s\times k}
		\end{pmatrix} ,
	\end{equation*}
	which the desired result.
\end{proof}

\begin{remark}
	If $m=n$, we have $(b, \ AV_{n})\in\mathbb{R}^{n\times(n+1)}$, and the property \eqref{3.10} can not be directly used. In fact, the procedure must terminate at step $n$, and thus $\beta_{n+1}=0$ and $u_{n+1}=0$. In this case, the form of \eqref{3.13} should be rewritten after some adjustments. Let $\hat{B}_{n}$ be the $n\times n$ lower bidiagonal form by discarding $\hat{\beta}_{n+1}$ and let $\hat{p}_{i}=\begin{pmatrix}
		-e_{i}^{(n)} \\
		\hat{u}_{i}
	\end{pmatrix} \in \mathbb{R}^{m+n}$ for $i=1,\dots,n$ and $P_{n+1}=I_{m+n}$. Then we can verify that 
	\begin{equation}\label{3.8}
		\begin{pmatrix}
			O_{n\times k} \\
			A\hat{V}_{k}
		\end{pmatrix} =
		\hat{P}_{1} \cdots \hat{P}_{k+1}\begin{pmatrix}
			\hat{B}_{k} \\
			O_{s\times k}
		\end{pmatrix} , 
	\end{equation}
	where $s=m+n-k-1$ for $k=1,\dots,n-1$ and $s=m$ for $k=n$.
\end{remark}


Now we give a corresponding version of Proposition \ref{prop3.1} in finite precision arithmetic. Similar to the above, we prove the result for $m>n$ and discuss the case of $m=n$ in the remark. The following lemma is needed, which is a generalization of \cite[Lemma 4.4]{Barlow2013} for one-sided reorthogonalizaton (for upper Lanczos bidiagonalization).
\begin{Lem}\label{lem3.1}
	For the $k$-th step LBRO, define the orthogonal matrix $\widehat{P}_{k+1}$ that is a product of Householder transformations as
	\begin{equation}\label{3.1}
		\widehat{P}_{k+1}=P_{1}\cdots P_{k+1} , \ \
		P_{i}=I_{m+n+1}-p_{i}p_{i}^{T}  , \ \
		p_{i}=\begin{pmatrix}
			-e_{i}^{(n+1)} \\
			u_{i}
		\end{pmatrix} \in \mathbb{R}^{m+n+1} ,
	\end{equation}
	then  we have
	\begin{equation}\label{3.4}
		\widehat{P}_{k+1}\begin{pmatrix}
			\alpha_{k}e_{k}^{(k)} \\
			\beta_{k+1}e_{1}^{(s+1)}
		\end{pmatrix} =
		\begin{pmatrix}
			0_{n+1} \\
			Av_{k}
		\end{pmatrix} +  x_{k}, \ \ s=m+n-k ,
	\end{equation}
	where $x_{k}\in \mathbb{R}^{m+n+1}$ and
	\begin{equation}\label{bnd_x}
		\lVert x_{k} \lVert = O(\lVert A \lVert (k\mathbf{u}+ k\nu_{k}+ \mu_{k+1})).
	\end{equation}
\end{Lem}
\begin{proof}
	By the definition of $P_{i}$ and together with \eqref{eq1} we have
	\begin{align*}
		P_{k+1}\begin{pmatrix}
			\alpha_{k}e_{k}^{(k)} \\
			\beta_{k+1}e_{1}^{(s+1)}
		\end{pmatrix}
		&= \begin{pmatrix}
			\alpha_{k}e_{k}^{(k)} \\
			\beta_{k+1}e_{1}^{(s+1)}
		\end{pmatrix} - p_{k+1}^{T}\begin{pmatrix}
			\alpha_{k}e_{k}^{(k)} \\
			\beta_{k+1}e_{1}^{(s+1)}
		\end{pmatrix}p_{k+1} \\
		&=\begin{pmatrix}
			\alpha_{k}e_{k}^{(k)} \\
			\beta_{k+1}e_{1}^{(s+1)}
		\end{pmatrix} + \beta_{k+1}\begin{pmatrix}
			-e_{k+1}^{(n+1)} \\
			u_{k+1}
		\end{pmatrix} \\
		&= \begin{pmatrix}
			\alpha_{k}e_{k}^{(n+1)} \\
			\beta_{k+1}u_{k+1}
		\end{pmatrix} \\
		&= \begin{pmatrix}
			\alpha_{k}e_{k}^{(n+1)} \\
			Av_{k}-\alpha_{k}u_{k}
		\end{pmatrix} -
		\begin{pmatrix}
			0_{n+1} \\
			U_{k}\bar{c}_{k}
		\end{pmatrix} -
		\begin{pmatrix}
			0_{n+1} \\
			f_{k}
		\end{pmatrix} .
	\end{align*}
	We also have
	\begin{align*}
		P_{k}\begin{pmatrix}
			\alpha_{k}e_{k}^{(n+1)} \\
			Av_{k}-\alpha_{k}u_{k}
		\end{pmatrix}
		&= P_{k}\begin{pmatrix}
			0_{n+1} \\
			Av_{k}
		\end{pmatrix}-\alpha_{k}P_{k}
		\begin{pmatrix}
			-e_{k}^{(n+1)} \\
			u_{k}
		\end{pmatrix} \\
		&= \begin{pmatrix}
			0_{n+1} \\
			Av_{k}
		\end{pmatrix} - (u_{k}^{T}Av_{k})p_{k} + \alpha_{k}p_{k} \\
		&=  \begin{pmatrix}
			0_{n+1} \\
			Av_{k}
		\end{pmatrix} - u_{k}^{T}(\alpha_{k}u_{k}+\beta_{k+1}u_{k+1}+U_{k}\bar{c}_{k}+f_{k})p_{k}
		+ \alpha_{k}p_{k} \\
		&=  \begin{pmatrix}
			0_{n+1} \\
			Av_{k}
		\end{pmatrix} - (\beta_{k+1}u_{k}^{T}u_{k+1}+u_{k}^{T}U_{k}\bar{c}_{k}+(u_{k}^{T}f_{k})p_{k} ,
	\end{align*}
	and
	\begin{equation*}
		P_{k}\begin{pmatrix}
			0_{n+1} \\
			U_{k}\bar{c}_{k}
		\end{pmatrix} =
		\begin{pmatrix}
			0_{n+1} \\
			U_{k}\bar{c}_{k}
		\end{pmatrix} - u_{k}^{T}U_{k}\bar{c}_{k}p_{k} , \ \
		P_{k}\begin{pmatrix}
			0_{n+1} \\
			f_{k}
		\end{pmatrix} =
		\begin{pmatrix}
			0_{n+1} \\
			f_{k}
		\end{pmatrix} - (u_{k}^{T}f_{k})p_{k} .
	\end{equation*}
	Therefore, we obtain
	\begin{equation}\label{3.5}
		P_{k}P_{k+1}\begin{pmatrix}
			\alpha_{k}e_{k}^{(k)} \\
			\beta_{k+1}e_{1}^{(s+1)}
		\end{pmatrix} =
		\begin{pmatrix}
			0_{n+1} \\
			Av_{k}
		\end{pmatrix} -\beta_{k+1}u_{k}^{T}u_{k+1}p_{k} -
		\begin{pmatrix}
			0_{n+1} \\
			U_{k}\bar{c}_{k}
		\end{pmatrix} -
		\begin{pmatrix}
			0_{n+1} \\
			f_{k}
		\end{pmatrix} .
	\end{equation}
	Let $\bar{w}_{k}=-\beta_{k+1}u_{k}^{T}u_{k+1}p_{k} -
	\begin{pmatrix}
		0_{n+1} \\
		U_{k}\bar{c}_{k}
	\end{pmatrix} -
	\begin{pmatrix}
		0_{n+1} \\
		f_{k}
	\end{pmatrix}$. Notice that $\|p_{k}\|=\sqrt{2}$. Using Proposition \ref{prop2.1} and the upper bound on $\beta_{k+1}$ in \eqref{bnd_beta}, we have
	$$\lVert \bar{w}_{k}\lVert=O(\lVert A \lVert(\mathbf{u}+\mu_{k+1}+\nu_{k})).$$
	For $i=1,\dots,k-1$, we have
	\begin{equation*}
		P_{i}\begin{pmatrix}
			0_{n+1} \\
			Av_{k}
		\end{pmatrix} =
		\begin{pmatrix}
			0_{n+1} \\
			Av_{k}
		\end{pmatrix} - (u_{i}^{T}Av_{k})p_{i} .
	\end{equation*}
	Using \eqref{2.21} we get
	\begin{align*}
		\lVert (u_{i}^{T}Av_{k})p_{i}\lVert
		&\leq \sqrt{2} \lVert v_{k}^{T}(V_i\tilde{d}_{i}+g_{i}) \lVert
		= \sqrt{2}\lVert (v_{k}^{T}V_i)\tilde{d}_{i} +v_k^{T}g_{i} \lVert \\
		&\leq \sqrt{2}[\nu_{k}(\lVert A \lVert  + O(\lVert A \lVert(\mathbf{u}+\nu_{i}))) + O(\lVert A \lVert\mathbf{u})] \\
		&= O(\lVert A \lVert(\mathbf{u}+\nu_{k}))
	\end{align*}
	or be written as 
	\begin{equation}\label{3.7}
		P_{i}\begin{pmatrix}
			0_{n+1} \\
			Av_{k}
		\end{pmatrix} =
		\begin{pmatrix}
			0_{n+1} \\
			Av_{k}
		\end{pmatrix} + \bar{w}_{i} , \ \ i=1,\dots,k-1,
	\end{equation}
	with $\bar{w}_{i}=- (u_{i}^{T}Av_{k})p_{i}$ and $\|\bar{w}_{i}\| = O(\lVert A \lVert(\mathbf{u}+\nu_{k}))$.
	
	Therefore, by \eqref{3.5} and \eqref{3.7}, we obtain
	\begin{align*}
		\widehat{P}_{k+1}\begin{pmatrix}
			\alpha_{k}e_{k}^{(k)} \\
			\beta_{k+1}e_{1}^{(s+1)}
		\end{pmatrix}
		= P_{1}\cdots P_{k-1}\Bigg( \begin{pmatrix}
			0_{n+1} \\
			Av_{k}
		\end{pmatrix} + \bar{w}_{k}   \Bigg) 
		= \begin{pmatrix}
			0_{n+1} \\
			Av_{k}
		\end{pmatrix} + x_{k} ,
	\end{align*}
	where $x_{k} = \sum_{i=1}^{k}(P_{1}\cdots P_{i-1})\bar{w}_{i}$. Notice that $P_i$ are Householder matrices and thus $\|P_i\|=1$. We finally obtain
	\begin{equation*}
		\lVert x_{k} \lVert
		\leq \sum_{i=1}^{k} \lVert \bar{w}_{i} \lVert
		= O(\lVert A \lVert (k\mathbf{u}+ k\nu_{k}+ \mu_{k+1})),
	\end{equation*}
	which is the desired result.
\end{proof}

\begin{remark}\label{remark3.2}
	For the one-sided reorthogonalization we have $\bar{c}_{k}=0$, which does not hold for a general reorthogonalization strategy. This makes the proof of \cite[Lemma 4.4]{Barlow2013} can not be applied to Lemma \ref{lem3.1}. The upper bound on $\bar{c}_{k}$ plays a key role in our proof. 
\end{remark}

The following theorem is a naturally corollary of Lemma \ref{lem3.1}, which generalizes \cite[Theorem 4.1]{Barlow2013}. The proof is similar and we omit it.
\begin{theorem}\label{thm3.1}
	For the $k$-step LBRO, we have
	\begin{equation}\label{3.2}
		\begin{pmatrix}
			O_{(n+1)\times k} \\
			AV_{k}
		\end{pmatrix} + X_{k} =
		\widehat{P}_{k+1}\begin{pmatrix}
			B_{k} \\
			O_{s\times k}
		\end{pmatrix} , \ \ s=m+n-k ,
	\end{equation}
	where $X_{k}=(x_{1}, \dots, x_{k})$ and
	\begin{equation}\label{3.3}
		\lVert X_{k} \lVert =
		O(\lVert A \lVert \sqrt{k}(k\mathbf{u}+k\nu_{k}+\mu_{k+1})) .
	\end{equation}
\end{theorem}

\begin{remark}\label{remark3.1}
	For $m=n$, let $B_{n}$ be the $n\times n$ lower bidiagonal form by discarding $\beta_{n+1}$. Then let $p_{i}=\begin{pmatrix}
		-e_{i}^{(n)} \\
		u_{i}
	\end{pmatrix} \in \mathbb{R}^{m+n}$ for $i=1,\dots,n$ and $P_{n+1}=I_{m+n}$. Now for $k<n$ \eqref{3.4} should be rewritten in the following form:
	\begin{equation*}
		\widehat{P}_{k+1}\begin{pmatrix}
			\alpha_{k}e_{k}^{(k)} \\
			\beta_{k+1}e_{1}^{(s+1)}
		\end{pmatrix} =
		\begin{pmatrix}
			0_{n} \\
			Av_{k}
		\end{pmatrix} +  x_{k} , \ \ s=m+n-k-1.
	\end{equation*}
	Especially, for $k=n$ \eqref{3.4} should be rewritten as
	\begin{equation*}
		\widehat{P}_{n+1}\begin{pmatrix}
			\alpha_{k}e_{n}^{(n)} \\
			0_m
		\end{pmatrix} =
		\begin{pmatrix}
			0_{n} \\
			Av_{n}
		\end{pmatrix} +  x_{n},
	\end{equation*}
	and the upper bound on $\|x_{n}\|$ should be
	\begin{equation}\label{bn_xn}
		\lVert x_{n} \lVert =
		O(\lVert A \lVert (n\mathbf{u}+n\nu_{n}+\mu_{n})).
	\end{equation}
	The result of Theorem \ref{thm3.1} can also be rewritten similarly. In order to obtain \eqref{bn_xn}, we first notice that $\beta_{n+1}u_{n+1}=Av_{n}-\alpha_{n}u_{n}-f_{n}$ since we do not need to reorthogonalize $u_{n+1}$. By letting $u_{n+1}=U_{n}l$ where $l\in\mathbb{R}^{n}$ and using methods similar to the proof of Proposition \ref{prop2.1}, we can get $\beta_{n+1} =  O(\lVert A \lVert(\mathbf{u}+\mu_{n}+\nu_{n}))$. Then \eqref{bn_xn} can be obtained with the help of the upper bound on $\beta_{n+1}$.
\end{remark}

Notice that Theorem \ref{thm3.1} is a corresponding version of Proposition \ref{prop3.1} in finite precision arithmetic. It establish a relationship between the $k$-step LBRO and Householder transformation based bidiagonal reduction of an augmented matrix of $AV_k$ with a perturbation.

\subsection{Mixed forward-backward error bound of the $k$-step LBRO}
There is a deep connection between the orthogonality level of $U_{k+1}$ and the detailed structure of $\widehat{P}_{k+1}$. In order to reveal it, we first state the following theorem which combines the results of \cite[Theorem 4.1]{Bjorck1992} and \cite[Theorem2.1 and Corralary 5.1]{Paige2009}.

\begin{theorem}[\cite{Bjorck1992,Paige2009}]\label{thm4.1}
	For any arbitrary integer $r \geq l \geq 1$, let $Q_{l}=(q_{1}, \dots, q_{l})\in \mathbb{R}^{r\times l}$ where $\lVert q_{j}\lVert=1$, $j=1,\dots,l$. Define
	\begin{equation*}
		W_{j}=I_{r+l}-w_{j}w_{j}^{T} , \ \
		w_{j} = \begin{pmatrix}
			-e_{j}^{(l)} \\
			q_{j}
		\end{pmatrix} \in \mathbb{R}^{r+l} ,
	\end{equation*}
	\begin{equation*}
		S_{l} = (I_{l}+M_{l})^{-1}M_{l} , \ \
		M_{l}=\mathbf{SUT}(Q_{l}^{T}Q_{l}) .
	\end{equation*}
	Then we have
	\begin{equation*}\label{4.1}
		W_{1}\cdots W_{l}=
		\begin{blockarray}{ccc}
			l & r & \\
			\begin{block}{(c|c)c}
				S_{l}& (I_{l}-S_{l})Q_{l}^{T} & l \\
				\BAhhline{|-|-|}
				Q_{l}(I_{l}-S_{l})& I_{n}-Q_{l}(I_{l}-S_{l})Q_{l}^{T} & r \\
			\end{block}
		\end{blockarray} ,
	\end{equation*}
	and
	\begin{equation}\label{4.2}
		\|S_{l}\| \leq 1, \ \ \
		\dfrac{\|M_{l}\|}{1+\|M_{l}\|}\leq\lVert S_{l} \lVert \leq 2\lVert M_{l}\lVert.
	\end{equation}
\end{theorem}

Notice that $\lVert M_{l}\lVert=\|\mathbf{SUT}(I_{l}-Q_{l}^{T}Q_{l}) \|$, which is just the orthogonality level of $Q_{l}$, the quantity $\|S_{l}\|\in[0,1]$ is an another beautiful measure of the orthogonality level of $Q_{l}$. The following result reveals a connection between the orthogonality level of $U_{k+1}$ and structure of $\widehat{P}_{k+1}$. For simplicity, we only prove the result for $m>n$, for $m=n$ the result can be proved similarly after some adjustments; see Remark \ref{remark3.1} for related discussions.

\begin{Lem}\label{lem4.1}
	For the $k$-step LBRO, there exist vectors $\tilde{u}_{k+2},\dots,\tilde{u}_{n+1}$ and $\tilde{v}_{k+2},\dots,\tilde{v}_{n}$ with $\|\tilde{u}_{i}\|=\|\tilde{v}_{i}\|=1$ and nonnegative numbers $\tilde{\alpha}_{k+2},\dots,\tilde{\alpha}_{n}$ and  $\tilde{\beta}_{k+2},\dots,\tilde{\beta}_{n+1}$, such that $\tilde{u}_{i}$ is orthogonal to $U_{k+1}$ and $\tilde{u}_{j}$, and $\tilde{v}_{i}$ is orthogonal to $V_{k+1}$ and $\tilde{v}_{j}$ for $i\neq j, i, j>k+1$, respectively, and for matrices  $\widetilde{U}=(U_{k+1},\tilde{u}_{k+2},\dots,\tilde{u}_{n+1})$, $\widetilde{V}=(V_{k+1},\tilde{v}_{k+2},\dots,\tilde{v}_{n})$ and
	$$\widetilde{B}=
	\begin{pmatrix}
		\alpha_{1} & & & & & & \\
		\beta_{2} & \ddots & & & & & \\
		& \ddots & \alpha_{k} & & & & \\
		& & \beta_{k+1} & \alpha_{k+1} & & & \\
		& & & \tilde{\beta}_{k+2} & \tilde{\alpha}_{k+2} & & \\
		& & & & \tilde{\beta}_{k+3} &  \ddots &  \\
		& & & & & \ddots & \tilde{\alpha}_{n} \\
		& & & & & & \tilde{\beta}_{n+1}
	\end{pmatrix} \in\mathbb{R}^{(n+1)\times n},$$
	the following properties hold. \\
	(1). There exist a matrix $\widetilde{X}\in\mathbb{R}^{(m+n+1)\times n}$ such that
	\begin{equation}\label{3.18}
		\begin{pmatrix}
			O_{(n+1)\times n} \\
			A\widetilde{V}
		\end{pmatrix} + \widetilde{X} =
		\widetilde{P}\begin{pmatrix}
			\widetilde{B} \\
			O_{m\times n}
		\end{pmatrix}
	\end{equation}
	and 
	\begin{equation}
		\lVert \widetilde{X}\lVert =
		O(\lVert A \lVert \sqrt{n}(k\mathbf{u}+k\nu_{k+1}+\mu_{k+1})),
	\end{equation}
	where 
	\begin{equation*}\label{3.20}
		\widetilde{P}= \widehat{P}_{k+1}\widetilde{P}_{k+2}\cdots\widetilde{P}_{n+1}, \ \ \ 
		\widetilde{P}_{i}=I_{m+n+1}-\tilde{p}_{i}\tilde{p}_{i}^{T}, \ \ \
		\tilde{p}_{i}=\begin{pmatrix}
			-e_{i}^{(n+1)} \\
			\tilde{u}_{i}
		\end{pmatrix}. 
	\end{equation*} \\
	(2). If $\widetilde{P}$ is partitioned as
	\begin{equation*}
		\widetilde{P} =
		\begin{blockarray}{ccc}
			n+1 & m & \\
			\begin{block}{(cc)c}
				\widetilde{P}_{11} & \widetilde{P}_{12} & n+1 \\
				\widetilde{P}_{21} & \widetilde{P}_{22} & m \\
			\end{block}
		\end{blockarray} ,
	\end{equation*}
	then we have
	\begin{equation}\label{3.19}
		\widetilde{P}_{21} = \widetilde{U}(I_{n+1}-\widetilde{P}_{11}), \ \ \ 
		\lVert \widetilde{P}_{11} \lVert \leq 2\mu_{k+1}.
	\end{equation}	 
\end{Lem}
\begin{proof}
	(1). We use the following procedure to construct vectors $\tilde{u}_{k+2},\dots,\tilde{u}_{n+1}$ and $\tilde{v}_{k+2},\dots,\tilde{v}_{n}$. After $k$ steps of the LBRO, we have computed $B_{k}$, $\alpha_{k+1}$, $U_{k+1}$ and $V_{k+1}$. At step $i\geq k+1$, vector $\tilde{u}_{i+1}$ is generated as
	$$r_i=A\tilde{v}_{i}-\alpha_{i}\tilde{u}_{i}, \ \ \ 
	\tilde{\beta}_{i+1}\tilde{u}_{i+1} =r_{i}-\sum\limits_{j=1}^{k+1}(u_{j}^{T}r_i)u_j-\sum\limits_{j=k+2}^{i}(\tilde{u}_{j}^{T}r_i)\tilde{u}_j$$
	such that $\|\tilde{u}_{i+1}\|=1$, where for $i=k+1$ we let $\tilde{u}_{k+1}=u_{k+1}$ and $\tilde{v}_{k+1}=v_{k+1}$. This procedure is equivalent to the one step Lanczos bidiagonalization of $\tilde{u}_{i+1}$ with full reorthogonalization. The construction of $\tilde{v}_{i+1}$ is similar, which is identical to the one step Lanczos bidiagonalization of $\tilde{v}_{i+1}$ with full reorthogonalization. If the procedure terminates at some step, it can be continued by choosing a new starting vector. Note that the above procedure can be treated as that we first perform the $k$-step LBRO in finite precision arithmetic and then for $i\geq k+1$ we perform the Lanczos bidiagonalization with full reorthogonalization in exact arithmetic. 
	
	By the above construction, we know that $\|\tilde{u}_{i}\|=\|\tilde{v}_{i}\|=1$ and $\tilde{u}_{i}$ is orthogonal to $U_{k+1}$ and $\tilde{u}_{j}$ while $\tilde{v}_{i}$ is orthogonal to $V_{k+1}$ and $\tilde{v}_{j}$ for $i\neq j, i, j>k+1$, respectively. Using this property, for $i\geq k+1$, by checking the proof of Lemma \ref{lem3.1} we can find that $\bar{w}_i=O(\lVert A \lVert (\mathbf{u}+\nu_{k+1}+ \mu_{k+1}))$, and $\bar{w}_j=0$ for $j=1,\dots,i-1$ except for $\bar{w}_j=O(\lVert A \lVert(\mathbf{u}+\nu_{k+1}))$ if $i=k+1$. 
	Therefore \eqref{3.20} becomes
	\[\widehat{P}_{k+1}\widetilde{P}_{k+2}\cdots\widetilde{P}_{i+1}\begin{pmatrix}
			\tilde{\alpha}_{i}e_{i}^{(i)} \\
			\tilde{\beta}_{i+1}e_{1}^{(s+1)}
		\end{pmatrix} =
		\begin{pmatrix}
			0_{n+1} \\
			A\tilde{v}_{i}
		\end{pmatrix} +  \tilde{x}_{i}, \ \ \ s=m+n-i \]
	with $\lVert \tilde{x}_{i} \lVert = O(\lVert A \lVert (k\mathbf{u}+ k\nu_{k+1}+ \mu_{k+1}))$ for $i=k+1$ or $\lVert \tilde{x}_{i} \lVert = O(\lVert A \lVert (\mathbf{u}+ \nu_{k+1}+ \mu_{k+1}))$ for $i>k+1$. Note that for $i=k+1$ we let $\tilde{v}_{k+1}=v_{k+1}$ and $\tilde{\alpha}_{k+1}=\alpha_{k+1}$. Therefore, equality \eqref{3.2} in Theorem \ref{thm3.1} becomes \eqref{3.18}, where
	$\widetilde{X}=(x_{1},\dots,x_{k},\tilde{x}_{k+1},\dots,\tilde{x}_{n})$ and thus
	\[\|\widetilde{X}\|\leq\sqrt{n}\max_{1\leq i \leq k \atop k+1\leq j\leq n}\{\|x_{i}\|,\|\tilde{x}_{j}\|\}
	=O(\lVert A \lVert \sqrt{n}(k\mathbf{u}+k\nu_{k+1}+\mu_{k+1})).\]
	~\\
	(2). Let $M=\mathbf{SUT}(\widetilde{U}^{T}\widetilde{U})$. By Theorem \ref{thm4.1}, we have
	\begin{equation*}
		\widetilde{P}_{11} = (I_{n+1}+M)^{-1}M, \ \
		\widetilde{P}_{21} = \widetilde{U}(I_{n+1}-\widetilde{P}_{11}) .
	\end{equation*}
	By inequality \eqref{4.2} we obtain 
	$$\lVert \widetilde{P}_{11} \lVert \leq 2\mu_{k+1},$$
	since the orthogonality level of $\widetilde{U}$ is $\|M\|=\mu_{k+1}$.
\end{proof}

Now we are ready to give a mixed forward-backward error bound of the $k$-step LBRO. We only prove the result for $m>n$, for the case of $m=n$ the result can be  proved similarly after some adjustments.
\begin{theorem}\label{thm4.2}
	For the $k$-step LBRO, there exist two orthornormal matrices $\bar{U}_{k+1}=(\bar{u}_{1},\dots,\bar{u}_{k+1})\in\mathbb{R}^{m\times(k+1)}$ and $\bar{V}_{k+1}=(\bar{v}_{1},\dots,\bar{v}_{k+1})\in \mathbb{R}^{n\times k}$ such that
	\begin{align}
		& \bar{U}_{k+1}(\beta_{1}e_{1}^{(k+1)}) = b+\delta_b , \label{4.14} \\
		& (A+E)\bar{V}_{k} = \bar{U}_{k+1}B_{k} , \label{4.15}  \\
		& (A+E)^{T}\bar{U}_{k+1} = \bar{V}_{k}B_{k}^{T}+\alpha_{k+1}\bar{v}_{k+1}(e_{k+1}^{(k+1)})^{T}, \label{4.16}
	\end{align}
	and 
	\begin{equation}\label{4.6}
		\lVert \bar{U}_{k+1}-U_{k+1} \lVert \leq 2\mu_{k+1} + O(\mu_{k+1}^{2}) , \ \ \ 
		\lVert \bar{V}_{k+1}-V_{k+1} \lVert \leq \nu_{k+1} +O(\nu_{k+1}^{2}) ,
	\end{equation}
	where $E$ and $\delta_b$ are perturbation matrix and vector, respectively, satisfying
	\begin{equation}\label{4.7}
		\lVert E \lVert =
		O(\lVert A \lVert \sqrt{n}(k\mathbf{u}+k\nu_{k+1}+\mu_{k+1})),
		\ \ \ \lVert \delta_b \lVert = O(\lVert b \lVert \mathbf{u}) .
	\end{equation}
\end{theorem}
\begin{proof}
	In finite precision arithmetic, we have $b + \delta_0 = \beta_{1}u_{1}$, where $\lVert \delta_0 \lVert = O(\lVert b \lVert \mathbf{u})$. 
	By the definition of $\widetilde{P}$, we have
	\begin{equation*}
		\begin{pmatrix}
			0_{n+1} \\
			b
		\end{pmatrix} + \delta_{1}
		= P_{1}
		\begin{pmatrix}
			\beta_{1} \\
			0_{m+n}
		\end{pmatrix} =
		\widetilde{P}
		\begin{pmatrix}
			\beta_{1} \\
			0_{m+n}
		\end{pmatrix} , \ \
		\delta_{1} =
		\begin{pmatrix}
			0_{n+1} \\
			\delta_0
		\end{pmatrix}.
	\end{equation*}
	Combining with \eqref{3.18}, we have
	\begin{equation}\label{4.9}
		\begin{pmatrix}
			0_{n+1} & O_{(n+1)\times n} \\
			b & A\widetilde{V}
		\end{pmatrix} +
		\begin{pmatrix}
			\widetilde{X}_{1} \\
			\widetilde{X}_{2}
		\end{pmatrix} =
		\begin{pmatrix}
			\widetilde{P}_{11} & \widetilde{P}_{12}  \\
			\widetilde{P}_{21} & \widetilde{P}_{22} \\
		\end{pmatrix}
		\begin{pmatrix}
			\bar{B} \\
			O_{m\times n}
		\end{pmatrix} =
		\begin{pmatrix}
			\widetilde{P}_{11}\bar{B} \\
			\widetilde{P}_{21}\bar{B}
		\end{pmatrix}
	\end{equation}
	where
	\begin{equation*}
		\begin{pmatrix}
			\widetilde{X}_{1} \\
			\widetilde{X}_{2}
		\end{pmatrix} =
		\begin{pmatrix}
			\delta_{0}, & \widetilde{X}
		\end{pmatrix} , \ \
		\bar{B} =
		\begin{pmatrix}
			\beta_{1}e_{n+1}^{(n+1)}, & \widetilde{B}
		\end{pmatrix} .
	\end{equation*}
	Since $\widetilde{P}$ is orthogonal, we have
	\begin{equation}\label{4.10}
		\widetilde{P}_{11}^{T}\widetilde{P}_{11}+\widetilde{P}_{21}^{T}\widetilde{P}_{21}=I_{n+1}.
	\end{equation}
	
	Now we construct two matrices $\bar{U}_{n+1}$ and $\bar{V}_{n}$. 
	Let the compact SVD of $\widetilde{P}_{21}$ is $\widetilde{P}_{21}=Y_{1}\Sigma Z^{T}$, where $Y=(Y_{1},Y_{2})\in \mathbb{R}^{m\times m}$ and $Z\in \mathbb{R}^{(n+1)\times (n+1)}$ are orthogonal matrices and $\Sigma =\mbox{diag}(\sigma_{1},\dots,\sigma_{n+1})$ satisfying $0\leq \sigma_{n+1} \leq\cdots\leq\sigma_{1}\leq 1$. Then by \eqref{4.10} we have
	\begin{equation*}
		Z^{T}\widetilde{P}_{11}^{T}\widetilde{P}_{11}Z=(I_{n+1}+\Sigma)(I_{n+1}-\Sigma) .
	\end{equation*}
	Define $\bar{U}_{n+1}=Y_{1}Z^{T} \in \mathbb{R}^{m\times (n+1)}$. Then $\bar{U}_{n+1}$ is an orthonormal matrix, and by the above equality we have
	\begin{equation}\label{4.11}
		\bar{U}_{n+1} - \widetilde{P}_{21}=
		Y_{1}(I_{n+1}-\Sigma)Z^{T} = Y_{1}(I_{n+1}+\Sigma)^{-1}Z^{T}\widetilde{P}_{11}^{T}\widetilde{P}_{11}.
	\end{equation}
	By \eqref{3.19} and \eqref{4.11} we obtain
	\begin{align*}
		\begin{split}
			\lVert \bar{U}_{n+1} - \widetilde{U}\lVert
			&= \lVert \bar{U}_{n+1} - \widetilde{P}_{21} - \widetilde{U}\widetilde{P}_{11} \lVert \\
			&\leq \lVert \bar{U}_{n+1} - \widetilde{P}_{21} \lVert + \lVert \widetilde{U}\widetilde{P}_{11} \lVert \\
			&\leq \lVert \widetilde{P}_{11}\lVert^{2} + \lVert \widetilde{P}_{11} \lVert(1+\mu_{k+1}) \\
			&\leq 4 \mu_{k+1}^{2} + 2\mu_{k+1}(1+\mu_{k+1}) \\
			&= 2\mu_{k+1} + O(\mu_{k+1}^{2}) .
		\end{split}
	\end{align*}
	By \eqref{4.9} and \eqref{4.11} we obtain
	\begin{align*}
		\bar{U}_{n+1}\bar{B} -
		\begin{pmatrix}
			b, &	A\widetilde{V}
		\end{pmatrix}
		= \bar{U}_{n+1}\bar{B} - (\widetilde{P}_{21}\bar{B}-\widetilde{X}_{2})
		= (\bar{U}_{n+1}-\widetilde{P}_{21})\bar{B}+\widetilde{X}_{2}
	\end{align*}
	and
	\begin{align*}
		(\bar{U}_{n+1}-\widehat{P}_{21})\bar{B}
		= Y_{1}(I_{n+1}+\Sigma)^{-1}Z^{T}\widehat{P}_{11}^{T}\widehat{P}_{11}\bar{B}
		= Y_{1}(I_{n+1}+\Sigma)^{-1}Z^{T}\widehat{P}_{11}^{T}\widetilde{X}_{1},
	\end{align*}
	which can be rewritten as
	\begin{equation}\label{error1}
		\bar{U}_{n+1}\bar{B} -
		\begin{pmatrix}
			b, &	A\widetilde{V}
		\end{pmatrix} =
		\begin{pmatrix}
			C, & I_{m}
		\end{pmatrix}
		\begin{pmatrix}
			\widetilde{X}_{1} \\
			\widetilde{X}_{1}
		\end{pmatrix} = \begin{pmatrix}
			C, & I_{m}
		\end{pmatrix}
		\begin{pmatrix}
			\delta_{0}, & \widetilde{X}
		\end{pmatrix} ,
	\end{equation}
	where $C=Y_{1}(I_{n+1}+\Sigma)^{-1}Z^{T}\widehat{P}_{11}^{T}$. Define 
	\begin{equation}\label{error2}
		\delta_{b} = \beta_{1}\bar{u}_{1}-b, \ \ \ 
		E_1 = \bar{U}_{n+1}\widetilde{B} - A\widetilde{V}.
	\end{equation}
	Then \eqref{error1} implies that
	\begin{align*}
		\delta_b =
		\begin{pmatrix}
			C, & I_{m}
		\end{pmatrix} \delta_{0}, \ \ \ 
		E_1 = \begin{pmatrix}
			C, & I_{m}
		\end{pmatrix}\delta X_{n}.
	\end{align*}
	Notice that
	$$CC^{T}+I_{m}=
	I_{m}+Y_{1}(I_{n+1}-\Sigma)(I_{n+1}+\Sigma)^{-1}Y_{1}^{T},$$
	and thus
	$\lVert
	\begin{pmatrix}
		C, & I_{m}
	\end{pmatrix} \lVert \leq \sqrt{2}$.
	Therefore,
	\begin{align*}
		& \lVert \delta_b \lVert \leq \lVert
		\begin{pmatrix}
			C, & I_{m}
		\end{pmatrix} \lVert \lVert \delta_{0} \lVert
		\leq \sqrt{2}\lVert \delta_{0} \lVert = O(\lVert b \lVert \mathbf{u})  ,  \\
		& \lVert E_1 \lVert \leq \lVert
		\begin{pmatrix}
			C, & I_{m}
		\end{pmatrix} \lVert \lVert \widetilde{X} \lVert
		\leq \sqrt{2}\lVert \widetilde{X} \lVert  .
	\end{align*}
	Let the SVD of $\widetilde{V}$ is $\widetilde{V}=K\Sigma_{1}J^{T}$. Define $\bar{V}_{n}=KJ^{T}\in \mathbb{R}^{n\times n}$. then $\bar{V}_{n}$ is orthogonal and
	\begin{equation*}
		\lVert \bar{V}_{n} - \widetilde{V} \lVert = \lVert I_{n+1} - \Sigma_{1} \lVert
		\leq 1- (1-2\nu_{k+1})^{1/2} = \nu_{k+1} +O(\nu_{k+1}^{2}) ,
	\end{equation*}
	where we have used $\sigma_{n}(\widetilde{V})\geq (1-2\nu_{k+1})^{1/2}$ since the orthogonality level of $\widetilde{V}$ is $\nu_{k+1}$; see \eqref{2.17}. By the definition of $E_{1}$ in \eqref{error1}, we have
	$A\widetilde{V}+E_1=\bar{U}_{n+1}\widetilde{B}$, which can be rewritten as
	\begin{equation}\label{equal1}
		(A+E)\bar{V}_{n} = \bar{U}_{n+1}\widetilde{B} \ \ \mbox{or} \ \
		(A+E)^{T}\bar{U}_{n+1} = \bar{V}_{n}\widetilde{B}^{T},
	\end{equation}	
	where
	$$E = [A(\widetilde{V}-\bar{V}_{n})+E_1]\bar{V}_{n}^{T},$$
	and thus
	$$\lVert E \lVert \leq \lVert A \lVert [\nu_{k+1} +O(\nu_{k+1}^{2})] +
	\sqrt{2}\lVert \widetilde{X} \lVert =O(\lVert A \lVert\sqrt{n} (k\mathbf{u}+k\nu_{k+1}+\mu_{k+1})) .$$
	Let $\bar{U}_{k+1}=\bar{U}_{n+1}(:,1:k+1)$ and $\bar{V}_{k+1}=\bar{V}_{n}(:,1:k+1)$. By using $\|\bar{U}_{k+1}-U_{k+1}\|\leq \lVert \bar{U}_{n+1} - \widetilde{U}\lVert$ and $\|\bar{V}_{k+1}-V_{k+1}\|\leq \lVert \bar{V}_{n} - \widetilde{V}\lVert$ we obtain \eqref{4.6}. By equating the first $k$ and $k+1$ columns of the two equalities of \eqref{equal1} respectively, we obtain \eqref{4.15} and \eqref{4.16}. Note that $\beta_{1}\bar{u}_{1}=\bar{U}_{k+1}(\beta_{1}e_{1}^{(k+1)})$. By \eqref{error2} we finally obtain \eqref{4.14}.
\end{proof}

Theorem \ref{thm4.2} is a generalization of \cite[Theorem 5.2]{Barlow2013} that deals with upper Lanczos bidiagonalization with one-sided reorthogonalization. There are three main improvements. First, our result can apply to a general reorthogonalization strategy, and by letting $\bar{c}_k=0_k$ for $k=1,2,\dots$ that corresponds to the one-sided reorthogonalization case we can obtain a similar upper bound on $\|E\|$ as \cite[Theorem 5.2]{Barlow2013}. Second, apart from giving an upper bound on $\|\bar{V}_{k+1}-V_{k+1}\|$, our result also give a similar upper bound on $\|\bar{U}_{k+1}-U_{k+1}\|$. Finally, although \cite[Theorem 5.2]{Barlow2013} is about the $k$-step process, upper bounds on $\|E\|$ and $\|\bar{V}_{k}-V_{k}\|$ there depend on the orthogonality level of $V_n$ that can only be known after the $n$-step process has finished, while our result, in contrast, can really apply to the $k$-step LBRO for $1\leq k \leq n$ due to the help of Lemma \ref{lem4.1}.

Notice that the relations \eqref{4.14}--\eqref{4.16} are matrix-form recurrences of the $k$-step Lanczos bidiagonalization of $A+E$ with starting vector $b+\delta_b$ (denoted by LB($A+E,b+\delta_{b}$)) in exact arithmetic. If the orthogonality levels of $U_{k+1}$ and $V_{k+1}$ are kept around $O(\mathbf{u})$, which corresponds to the full reorthogonalization case, then $B_k$ is the exact one generated by the $k$-step LB($A+E,b+\delta_{b}$) with $\|E\|/\|A\|, \|\delta_{b}\|/\|b\|=O(\mathbf{u})$. Following Higham \cite[\S 1.5] {Higham2002}, we could call the $k$-step LBRO \textit{mixed forward-backward stable} as long as the orthogonality of $U_{k+1}$ and $V_{k+1}$ are good enough \footnote{The mixed forward-backward stability can be illustrated by the following example. Suppose a method is used to compute $y=f(x)$ for $f:\mathbb{R}\rightarrow \mathbb{R}$. Then the method is called mixed forward-backward stable if the computed value $\widehat{y}$ satisfies $\widehat{y}+\Delta y=f(x+\Delta x)$ where $|\Delta y|/|y|$ and $|\Delta x|/|x|$ are sufficiently small.}. This result can be used to analyze backward stability of LBRO based algorithms, such as bidiagonalization based algorithms for computing a partial SVD or LSQR for iteratively solving least squares problems.

\section{Applications to SVD computation and LSQR}\label{sec4}
In this section, we use the above results to investigate backward stability of LBRO based algorithms including partial SVD computation and LSQR. For simplicity, we only consider rounding errors in the LBRO, while other parts of an algorithm is supposed to be performed in exact arithmetic.

We first review the partial SVD computation of $A$ based on the Lanzos bidiagonalization in exact arithmetic. We add `` $\hat{}$ " to $\alpha_{i}$, $u_{i}$, $B_{k}$, etc. to denote the corresponding quantities in exact arithmetic. In order to approximate some singular values and corresponding vectors of $A$, we first compute the compact SVD of $\hat{B}_{k}$:
\begin{equation}\label{5.1}
	\hat{B}_{k} = H_{k}\Theta_{k}Z_{k}^{T}, \ \ \Theta_{k}=\mbox{diag}(s_{1}^{(k)}, \dots, s_{k}^{(k)}), \ \ 
	s_{1}^{(k)} > \dots > s_{k}^{(k)} > 0 ,
\end{equation}
where $H_{k}=(h_{1}^{(k)}, \dots, h_{k}^{(k)})\in\mathbb{R}^{(k+1)\times k}$ and $Z_{k}=(z_{1}^{(k)}, \dots, z_{k}^{(k)})\in\mathbb{R}^{k\times k}$ are orthonormal, and $\Theta_{k}\in\mathbb{R}^{k\times k}$. Then the approximate singular values of $A$ are $s_{i}^{(k)}$, which are often called Ritz values, and the corresponding approximate left and right singular vectors are $x_{i}^{(k)}=\hat{U}_{k+1}h_{i}^{(k)}$ and $y_{i}^{(k)}=\hat{V}_{k}z_{i}^{(k)}$, respectively, where $i=1,\dots,k$. The above decomposition can be achieved by a direct method since $\hat{B}_{k}$ is a matrix of small scale, and it is appropriate to compute some extreme singular values and vectors \cite{Larsen1998}.


In finite precision arithmetic, the above procedure is suffered from the ``ghosts" phenomenon, which means that some of the converged Ritz values suddenly ``jump" to become a ghost and then converge to the next larger or smaller singular values after a few iterations. These redundant copies of Ritz values are usually called ``ghosts" and this phenomenon results to many unwanted spurious copies of singular values and make it difficult to determine whether these spurious copies are real multiple singular values. This problem is closely related to the loss of orthogonality of Lanczos vectors, which can be avoided by using some types of reorthogonalization strategies \cite{Larsen1998,Simon2000}. The full reorthogonalization strategy is often used to mimic the Lanczos bidiagonalization in exact arithmetic. In the following we give an analysis of
full reorthogonalization for partial SVD computation in finite precision arithmetic.


Suppose the Lanczos bidiagonalization is implemented with full reorthogonalization and the orthogonality levels of $U_{k+1}$ and $V_{k+1}$ are kept around $O(\mathbf{u})$. By Theorem \ref{thm4.2}, $s_{i}^{(k)}$ are Ritz values of $\widetilde{A}=A+E$ and they will converge to the singular values of $\widetilde{A}$. Notice that $\|E\|=O(\|A\|\sqrt{n}k\mathbf{u})$. By the perturbation theory of singular values we have
\begin{equation}\label{5.2}
	\dfrac{|\sigma_{i}(A)-\sigma_{i}(\widetilde{A})|}{\sigma_{i}(A)}=\dfrac{|\sigma_{i}(A)-\sigma_{i}(\widetilde{A})|}{\|A\|}\cdot \dfrac{\sigma_{1}(A)}{\sigma_{i}(A)}\leq  \dfrac{\sigma_{1}(A)}{\sigma_{i}(A)}\cdot O(\sqrt{n}k\mathbf{u})
\end{equation}
for $\sigma_{i}(A)>0$. Therefore, the several largest Ritz values $s_{i}^{(k)}$ will approximate largest singular values of $A$ with relative error around $O(\mathbf{u})$. If $\sigma_{i}(A)$ has multiplicity bigger than one, then the several converged Ritz values corresponding to $\sigma_{i}(A)$ are not strictly equal. By \eqref{5.2}, with the help of full reorthogonalization, errors between these Ritz values will be so small that they all look like approximations to the same $\sigma_{i}(A)$, but they actually differ by a value of $O(\sqrt{n}k\mathbf{u})$. If the algorithm is implemented in a lower precision arithmetic, such as single precision instead of double precision, then separations between those converged Ritz values corresponding to a multiple singular value will be observed more obviously. For a very small singular value $\sigma_{i}(A)$ that has multiplicity bigger than one, the converged Ritz values corresponding to $\sigma_{i}(A)$ will have bigger relative errors since $\sigma_{1}(A)/\sigma_{i}(A)$ is very big. The above discussion will be illustrated by using an numerical example later.

Now we investigate the LSQR that is based on the Lanczos bidiagonalization for iteratively solving a large scale linear least squares problem
\begin{equation}
	\min\limits_{x \in \mathbb{R}^{n}} \lVert Ax - b \lVert .
\end{equation}
In exact arithmetic, the problem is projected to the Krylov subspace $\mathcal{K}_{k}(A^{T}A,A^{T}b)=\mathrm{span}(\hat{V}_{k})$ to become a $k$-dimension small scale problem. Denoting a vector in $\mathrm{span}(\hat{V}_{k})$ by $x=\hat{V}_{k}y$ where $y\in\mathbb{R}^{k}$, we have
\begin{align*}
	\begin{split} \ \ \ \ \min\limits_{x=\hat{V}_{k}y} \lVert Ax - b \lVert
		&= \min\limits_{y\in \mathbb{R}^{k}} \lVert A\hat{V}_{k}y -
		\hat{\beta}_{1}U_{k+1}e_{1}^{(k+1)}\lVert \\
		&= \min\limits_{y\in \mathbb{R}^{k}} \lVert \hat{U}_{k+1}\hat{B}_{k}y -
		\hat{\beta}_{1}\hat{U}_{k+1}e_{1}^{(k+1)}\lVert \\
		&= \min\limits_{y\in \mathbb{R}^{k}} \lVert \hat{B}_{k}y -\hat{\beta}_{1}e_{1}^{(k+1)}\lVert ,
	\end{split}
\end{align*}
and thus the $k$-th iterative solution of the LSQR is
$$\hat{x}_{k}=\hat{V}_{k}\hat{y}_{k} , \ \
	\hat{y}_{k}=\mathrm{arg}\min\limits_{y\in \mathbb{R}^{k}} \lVert \hat{B}_{k}y -\hat{\beta}_{1}e_{k+1}^{(k+1)}\lVert.$$
There is a formula that can recursively update $\hat{x}_{k}$ to $\hat{x}_{k+1}$ very efficiently, which avoids solving the $k$-dimension small-scale problem at each iteration \cite{Bjorck1988}. 

In finite precision arithmetic, the convergence of LSQR can be slowed down significantly due to the loss of orthogonality of Lanczos vectors. Maintaining a certain level of orthogonality among Lanczos vectors will accelerate the convergence at the expense of a bit more computational costs and storage requirements \footnote{In full reorthogonalization, $u_k$ and $v_k$ are reorthogonalized against all previous vectors $\{u_1,\dots, u_{k-1}\}$ and $\{v_1,\dots, v_{k-1}\}$ as soon as they have been computed. This adds an arithmetic cost of about $4(m + n)k^2$ flops, which is affordable if $k \ll\min\{m,n\}$.}. For example, 
the LSQR is usually implemented with full reorthogonalization to maintain stability of convergence for solving discrete linear ill-posed problems, since it usually takes not too many iterations to reach the semi-convergence point \cite{Hnetyn2009}. 

We analyze the LSQR with full reorthogonalization for solving least squares problems. In finite precision arithmetic, the $k$-th iteration LSQR computes
\begin{equation*}
	x_{k} = V_{k}y_{k} , \\ \ \
	y_{k} = \mathrm{arg}\min\limits_{y\in \mathbb{R}^{k}} \lVert B_{k}y - \beta_{1}e_{1}^{(k+1)} \lVert .
\end{equation*}
By \eqref{4.14} and \eqref{4.15}, we have
\begin{align*}
	\begin{split}
		y_{k}
		= \mathrm{arg}\min\limits_{y\in \mathbb{R}^{k}} \lVert \bar{U}_{k+1}B_{k}y - \bar{U}_{k+1}\beta_{1}e_{1}^{(k+1)} \lVert 
		= \mathrm{arg}\min\limits_{y\in \mathbb{R}^{k}} \lVert (A+E)\bar{V}_{k}y-(b+\delta_b) \lVert .
	\end{split}
\end{align*}
Therefore, we get $\bar{x}_{k} = \mathrm{arg}\min\limits_{x = \bar{V}_{k}y} \lVert (A+E)x- (b+\delta_b) \lVert$ where $\bar{x}_{k}=\bar{V}_{k}y_{k}$, which implies that
$\bar{x}_{k}$ is the $k$-th step LSQR solution in exact arithmetic to the perturbed problem
\begin{equation}\label{5.7}
	\min\limits_{x \in \mathbb{R}^{n}} \lVert (A+E)x-
	(b+\delta_b) \lVert .
\end{equation}
Notice from \eqref{4.6} that 
$$\lVert \bar{x}_{k}-x_{k}\lVert \leq \|V_{k}-\bar{V}_{k}\|\|y_{k}\|\leq  \|y_{k}\|(\nu_{k}+O(\nu_{k}^{2})).$$
Using $\|y_{k}\|=\|\bar{x}_{k}\|$ we obtain
\begin{equation}
	\dfrac{\|x_{k}-\bar{x}_{k}\|}{\|\bar{x}_{k}\|}\leq \nu_{k}+O(\nu_{k}^{2}).
\end{equation}
If we keep the orthogonality of $V_{k}$ around $O(\mathbf{u})$, then $x_{k}\approx \bar{x}_{k}$ and thus $x_{k}$ can be treated as the $k$-step LSQR solution to the perturbed problem \eqref{5.7}. This indicates the backward stability property of LSQR if full reorthogonalization is used.

\section{Numerical experiments}\label{sec5}
In this section, we use several numerical examples to illustrate our results. Since the error matrix $\|E\|$ in Theorem \ref{thm4.2} can not be explicitly calculated, we only show the values of $\|X_{k}\|$ in Theorem \ref{thm3.1}. In order to confirm the upper bound on $\|E\|$ implicitly, we use an example to illustrate the relation \eqref{5.2} for partial SVD computations. The following four matrices are taken from \cite{Davis2011}, and a description of them is in Table \ref{tab1}, where $\kappa(\cdot)$ is the condition number of a matrix. All the computations are carried out in MATLAB R2019b, where the roundoff unit for double precision arithmetic is $\mathbf{u}=2^{-53}\approx 1.11\times 10^{-16}$. All starting vectors are chosen as $b=(1,\dots,1)\in\mathbb{R}^{m}$.

\begin{table}[htp]
	\centering
	\caption{Properties of the test matrices.}
	\begin{tabular}{l|lllll}
		\toprule
		Index & $A$	 &$m\times n$		 &$\|A\|$	&$\kappa(A)$  & Description	 \\   
		\hline
		1 & {\sf nos3}	 & $960\times 960$	& 689.904 & 37723.6 & structural problem	\\  
		2 & {\sf well1850}  &$1850\times 712$   & 1.79433 & 111.313 & least squares problem  \\ 
		3 & {\sf lshp2614}    &$2614\times 2614$  & 6.98798  & 5197.35  & thermal problem  \\  
		4 & {\sf c-23} &$3969\times 3969$ & 1089.71 & 22795.9 & optimization problem  \\  \bottomrule
	\end{tabular}
	\label{tab1}
\end{table}

\begin{figure}[htp]
	\centering
	\begin{subfigure}[t]{0.35\textwidth}
		\centering
		\includegraphics[width=\textwidth]{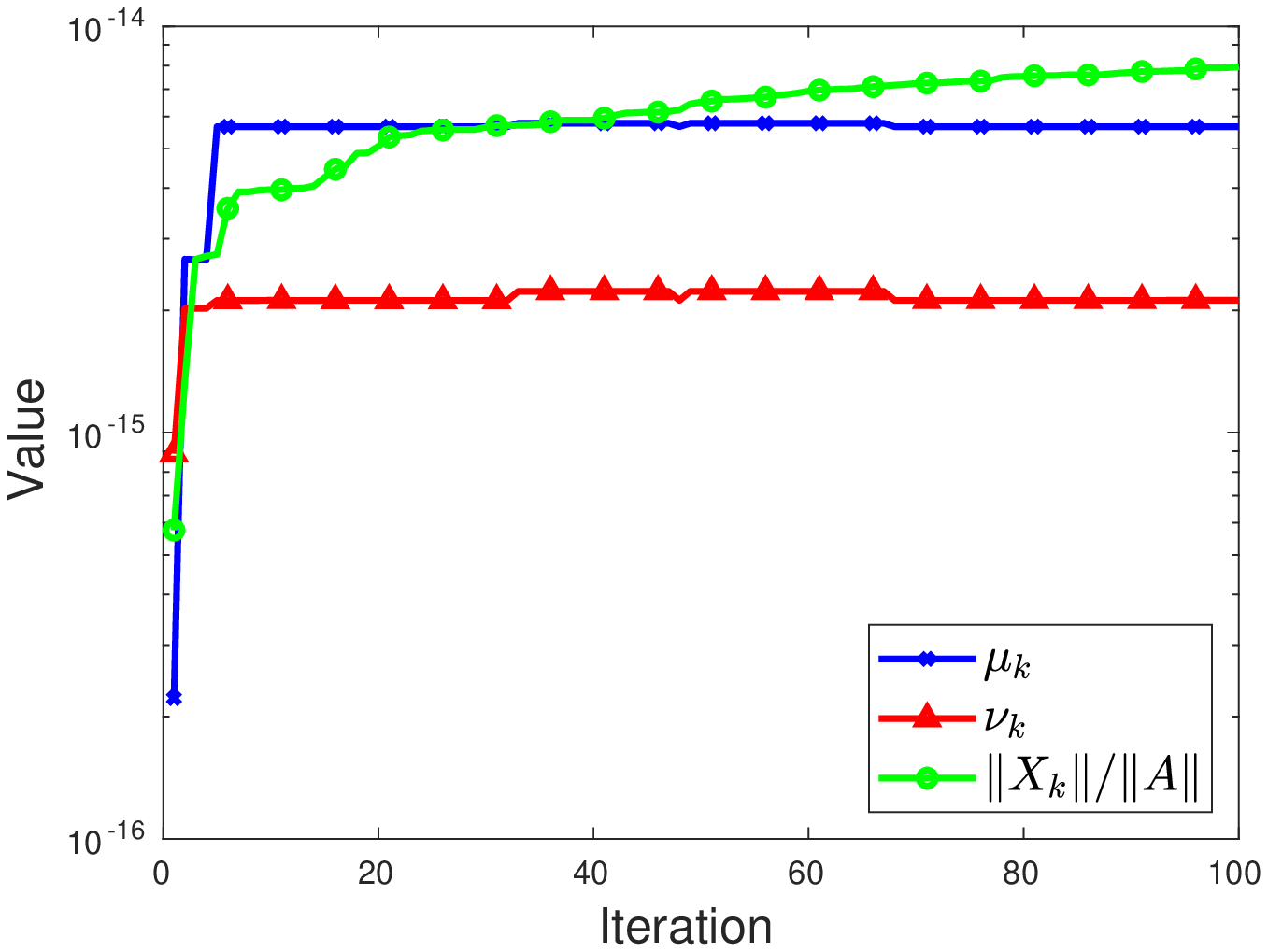}
		\caption{}
	\end{subfigure}
	\begin{subfigure}[t]{0.35\textwidth}
		\centering
		\includegraphics[width=\textwidth]{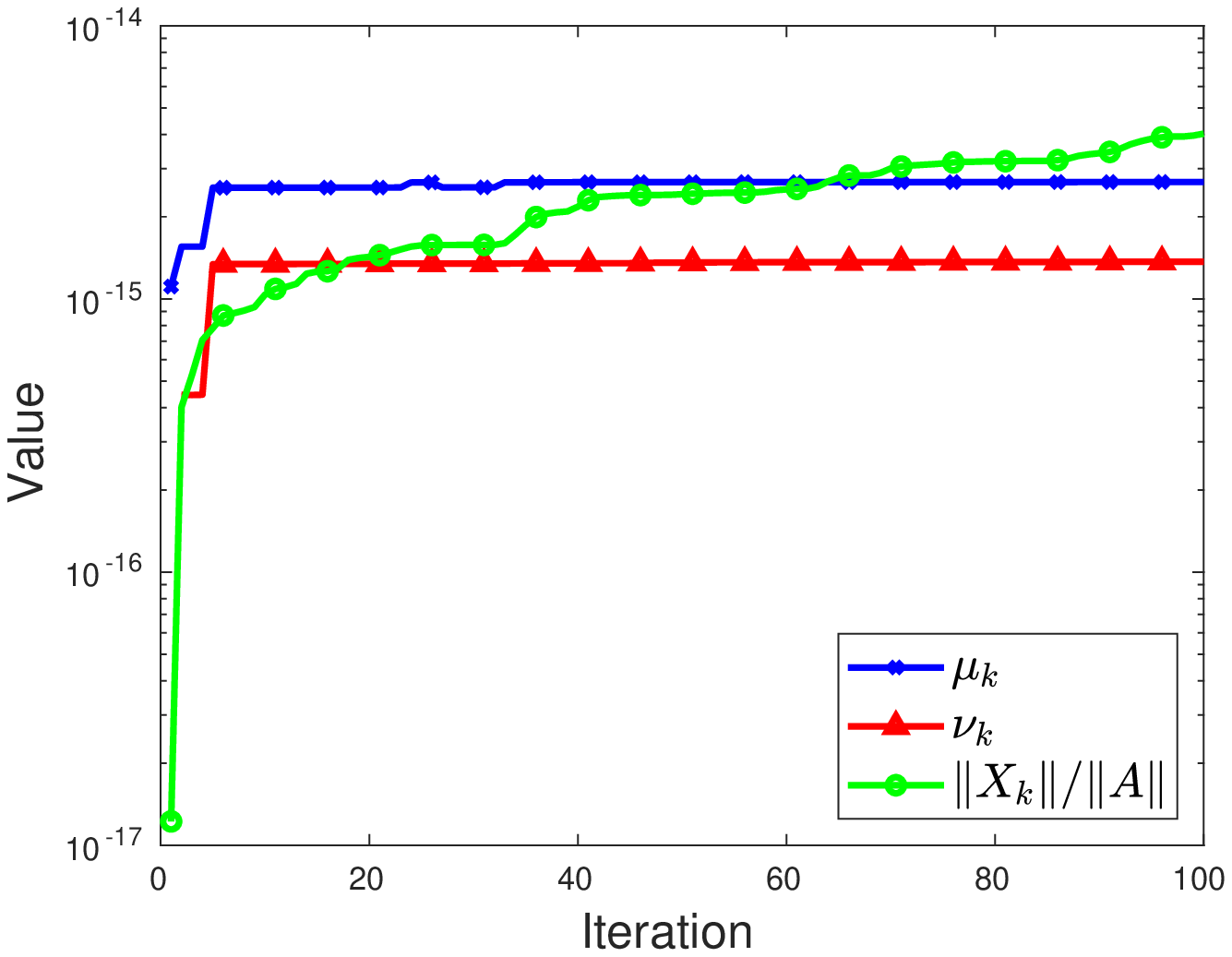}
		\caption{}
	\end{subfigure}
	\begin{subfigure}[t]{0.35\textwidth}
		\centering
		\includegraphics[width=\textwidth]{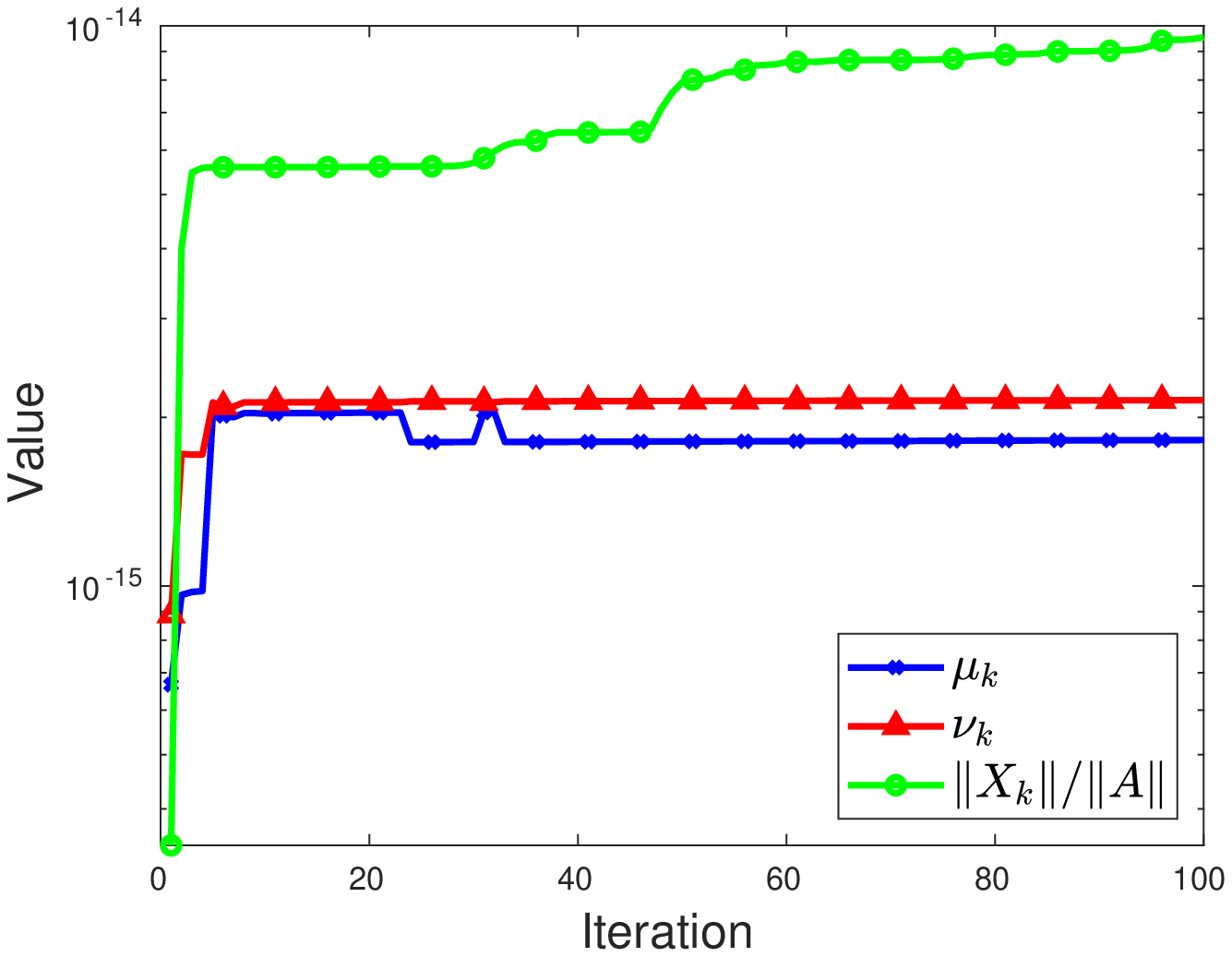}
		\caption{}
	\end{subfigure}
	\begin{subfigure}[t]{0.35\textwidth}
		\centering
		\includegraphics[width=\textwidth]{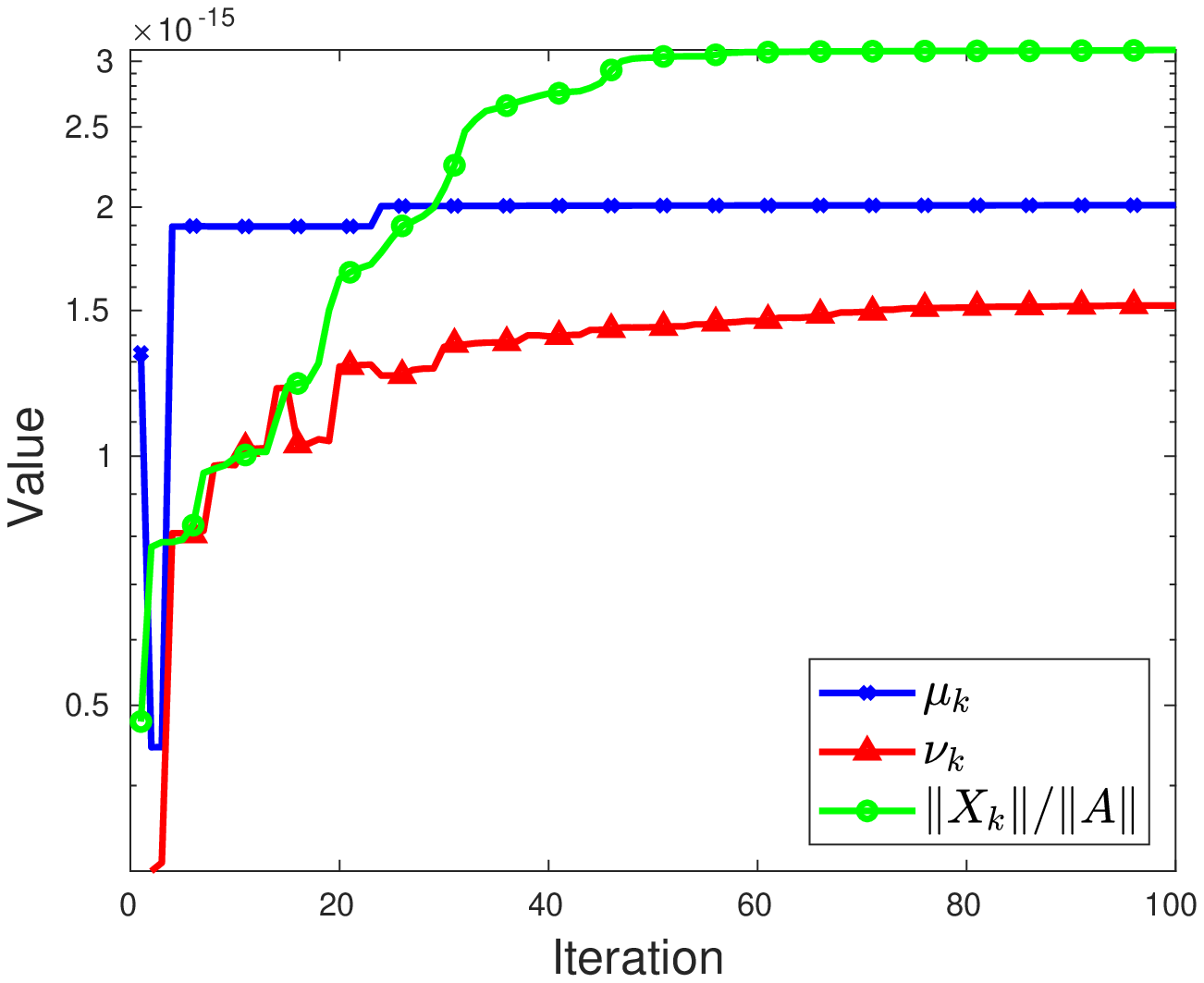}
		\caption{}
	\end{subfigure}
	\caption{Values of $\|{X}_{k}\|/\|A\|$ for full reorthogonalization: (a) {\sf nos3}; (b) {\sf well1850}; (c) {\sf lshp2614}; (d) {\sf c-23}.}
	\label{fig1}
\end{figure}

Figure \ref{fig1} depicts the values of $\|{X}_{k}\|/\|A\|$ and orthogonality levels of $U_k$ and $V_k$ as the iterations progress from $1$ to $100$, where the LBRO is implemented using full reorthogonalization. For the four matrices, the orthogonality levels of $U_k$ and $V_k$ are maintained around $O(\mathbf{u})$. Therefore, by the upper bound in \eqref{3.3}, $\|{X}_{k}\|/\|A\|$ should be around $O(\mathbf{u})$, which can be observed from the four subfigures.

\begin{figure}[htp]
	\centering
	\begin{subfigure}[t]{0.35\textwidth}
		\centering
		\includegraphics[width=\textwidth]{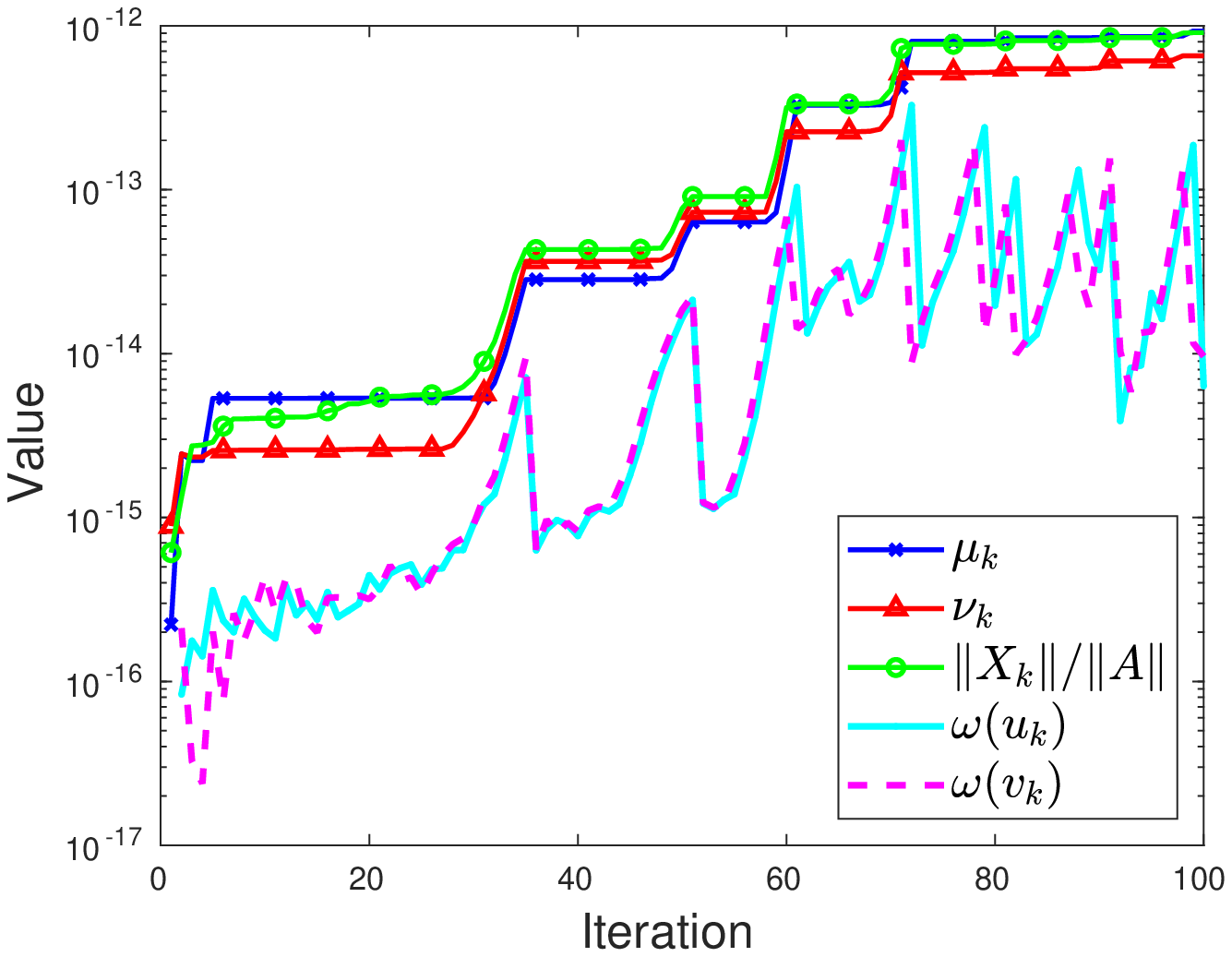}
		\caption{}
	\end{subfigure}
	\begin{subfigure}[t]{0.35\textwidth}
		\centering
		\includegraphics[width=\textwidth]{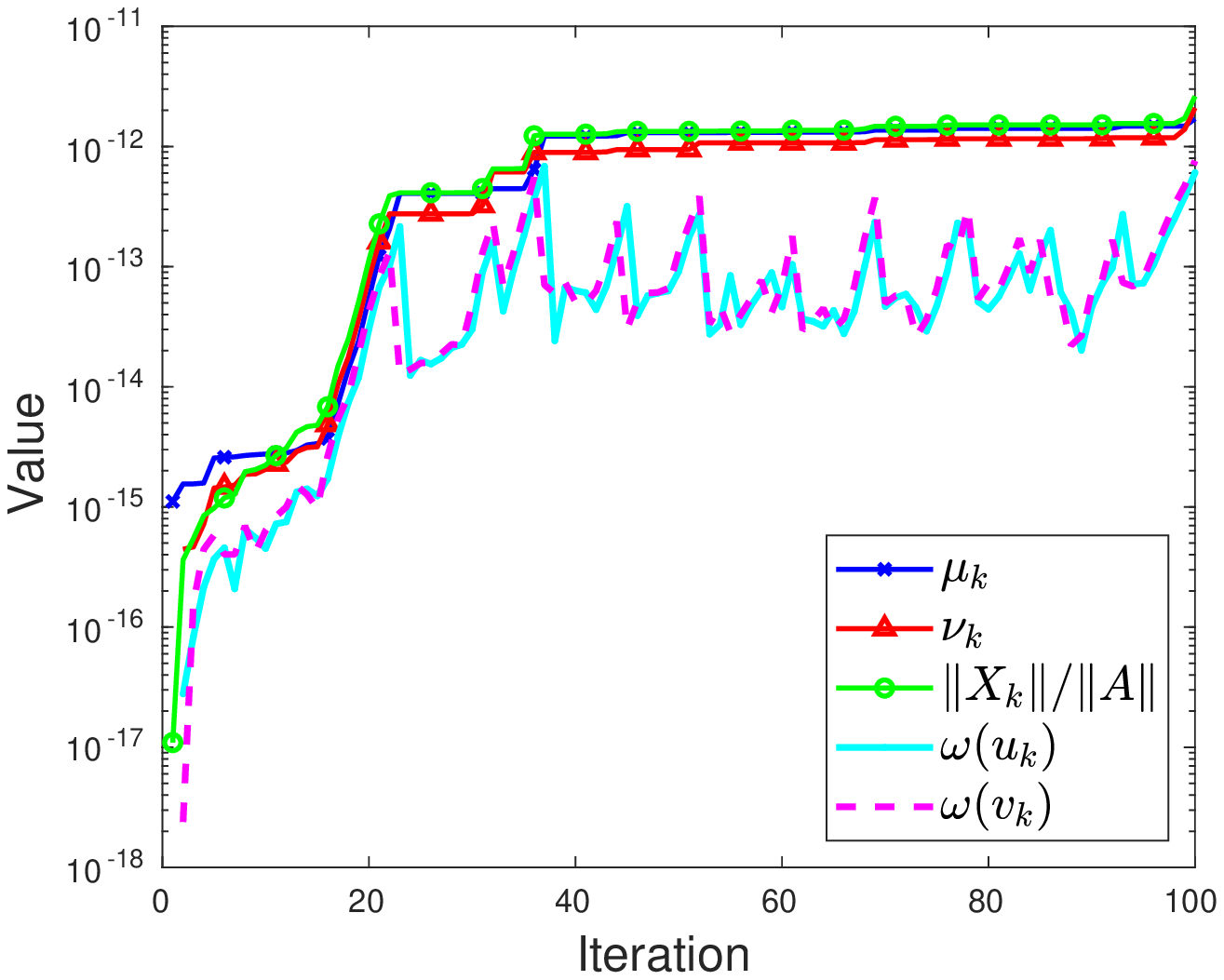}
		\caption{}
	\end{subfigure}
	\begin{subfigure}[t]{0.35\textwidth}
		\centering
		\includegraphics[width=\textwidth]{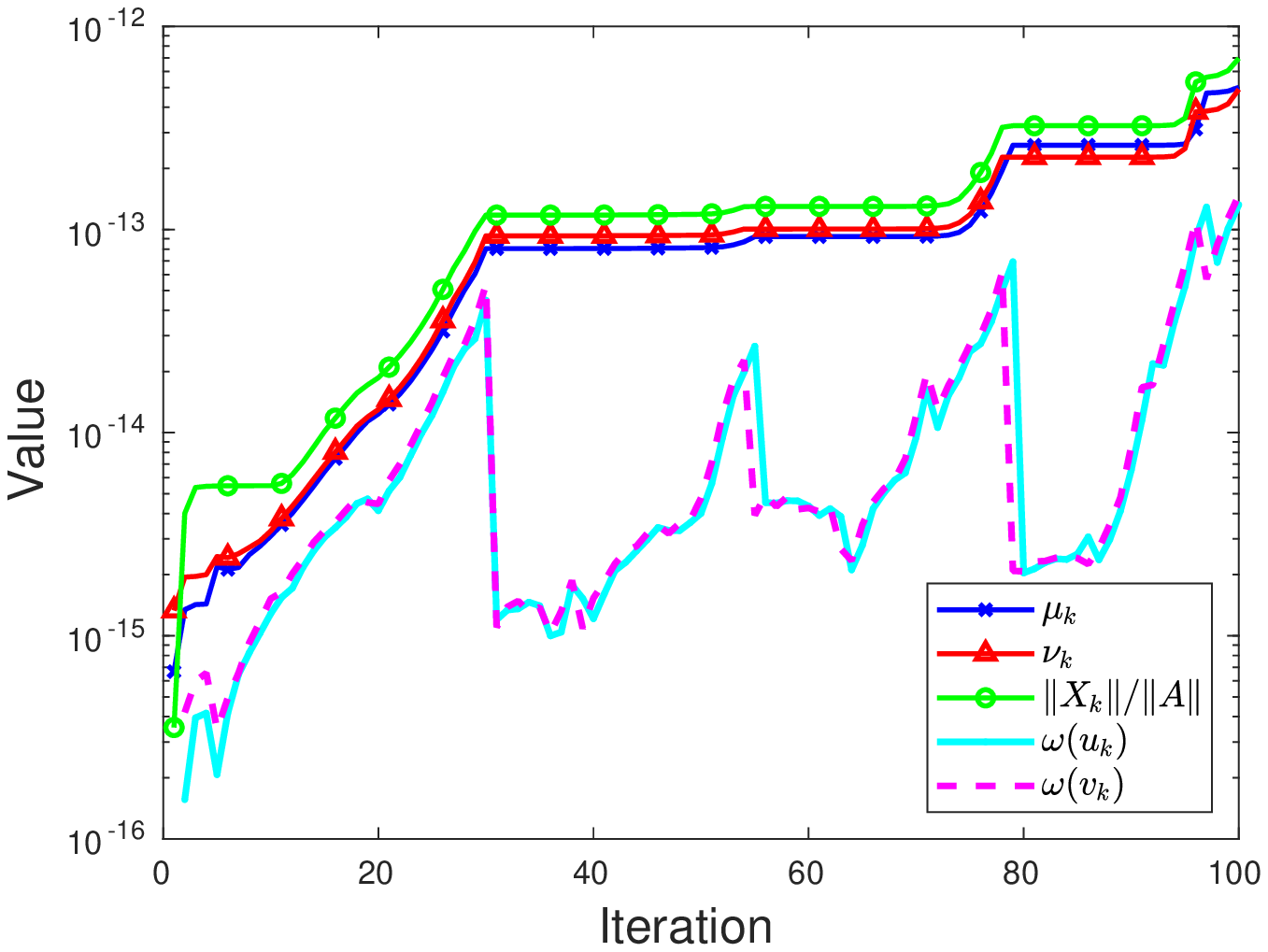}
		\caption{}
	\end{subfigure}
	\begin{subfigure}[t]{0.35\textwidth}
		\centering
		\includegraphics[width=\textwidth]{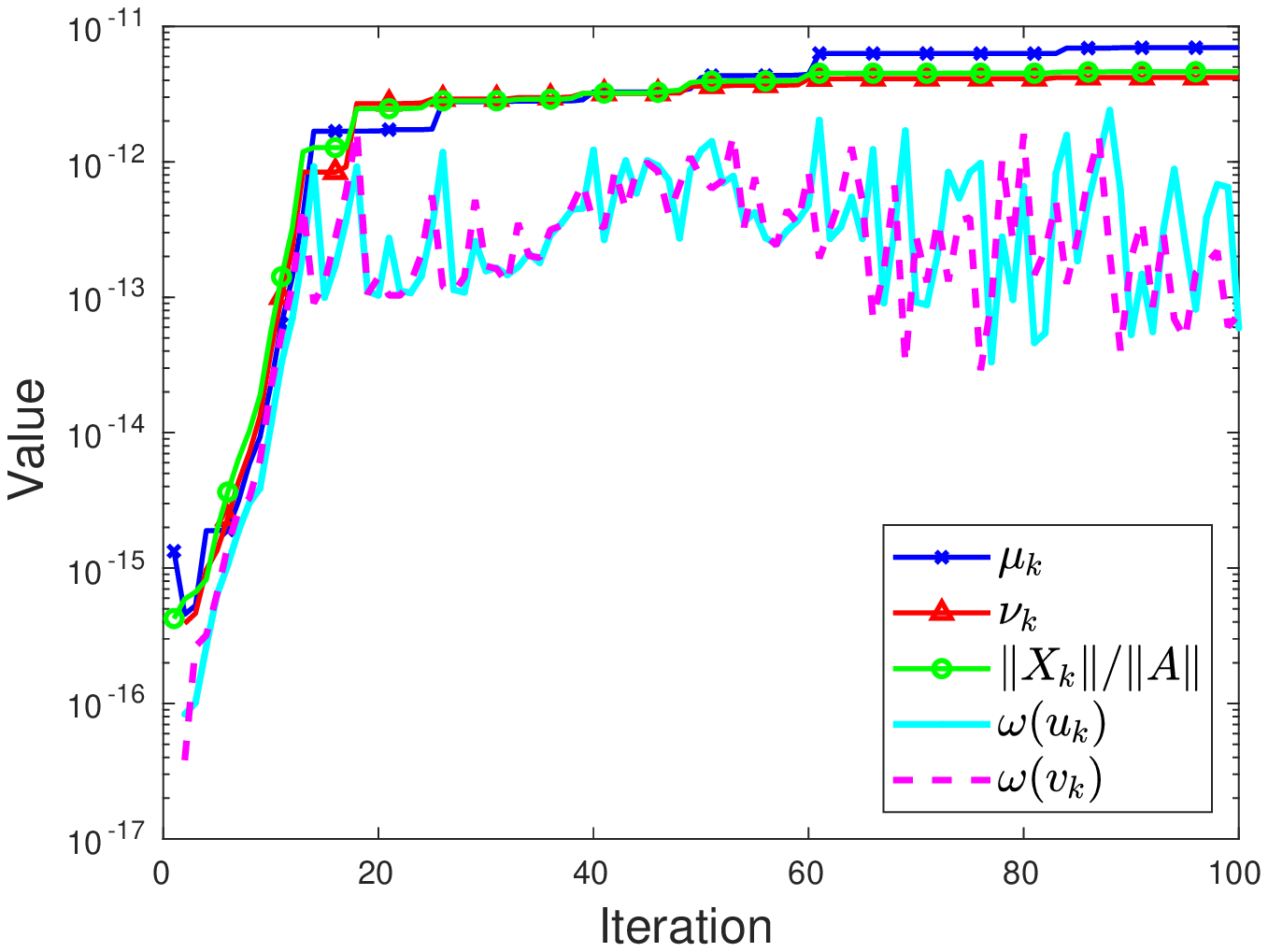}
		\caption{}
	\end{subfigure}
	\caption{Values of $\|{X}_{k}\|/\|A\|$ for partial reorthogonalization: (a) {\sf nos3}; (b) {\sf well1850}; (c) {\sf lshp2614}; (d) {\sf c-23}.}
	\label{fig2}
\end{figure}

Then we use partial reorthogonalization to implement the LBRO by using \texttt{lanbpro.m} in the \texttt{PROPACK} package \cite{Larsen1998}. Figure \ref{fig2} depicts the values of $\|{X}_{k}\|/\|A\|$ as well as $\mu_k$ and $\nu_k$. We also depict the orthogonality levels of $u_k$ and $v_k$ defined as
$$\omega(u_{k})=\max_{1\leq i< k}|u_{i}^{T}u_{k}|, \ \ \
	\omega(v_{k})=\max_{1\leq i< k}|v_{i}^{T}v_{k}|,$$
respectively. Note that $\omega(u_{i})\leq \mu_{k}\leq k\max_{1\leq i \leq k}\omega(u_{i})$ and this quantity is used to measure orthogonality between $u_k$ and previous Lanczos vectors. From the figure we find that the values of $\mu_k$ and $\nu_k$ grow slowly and eventually stabilize around a value.  This is because we set \texttt{OPTIONS.eta=1e-10} in \texttt{lanbpro.m}, which can keep orthogonality levels below $10^{-10}$ after reorthogonalization. We can also find that $\|{X}_{k}\|/\|A\|$ grows in consistent with $\mu_k$ and $\nu_k$, which confirms the upper bound \eqref{3.3} again. We find that $\omega(u_{k})$ and $\omega(v_{k})$ grow gradually at first, then they suddenly jump down and reorthogonalization is not used in a few later steps until the orthogonality becomes bad again. This is because partial reorthogonalizations occurs only when necessary and only a part of previous Lanczos vectors are included in the reorthogonalization steps. 

Finally we use an example to illustrate the relation \eqref{5.2}. The matrix $A$ is constructed as follows. Let $m=n=800$. First construct a row vector $s$ such that $s(1) = s(2) = 1.0, s(3)=0.95, s(n-2)=0.1,s(n-1)=s(n) = 10^{-4}$ and $\texttt{s(4:n-3) = linspace(0.90,0.15,n-6)}$ generated by the MATLAB built-in function \texttt{linspace()}, and then let $\texttt{S=diag(s)}$. Let $P$ and $Q$ be two symmetric orthogonal matrices generated by the MATLAB built-in functions $\texttt{P = gallery(`orthog',n,1)}$ and $\texttt{Q = gallery(`orthog',n,2)}$, respectively. Finally let $A=PSQ^{T}$. By the construction,  the $i$-th largest singular value of $A$ is $s_{i}$, and the multiplicities of singular values $\sigma_{1}=1.0$ and $\sigma_{n}=10^{-4}$ is $2$. In the implementations, we first store both $A$ and $b$ in double precision, and then the LBRO using full reorthogonalization is implemented in double and single precision arithmetic, respectively, and all the other computations are taken in double precision.

\begin{figure}[htp]
	\centering
	\begin{subfigure}[t]{0.35\textwidth}
		\centering
		\includegraphics[width=\textwidth]{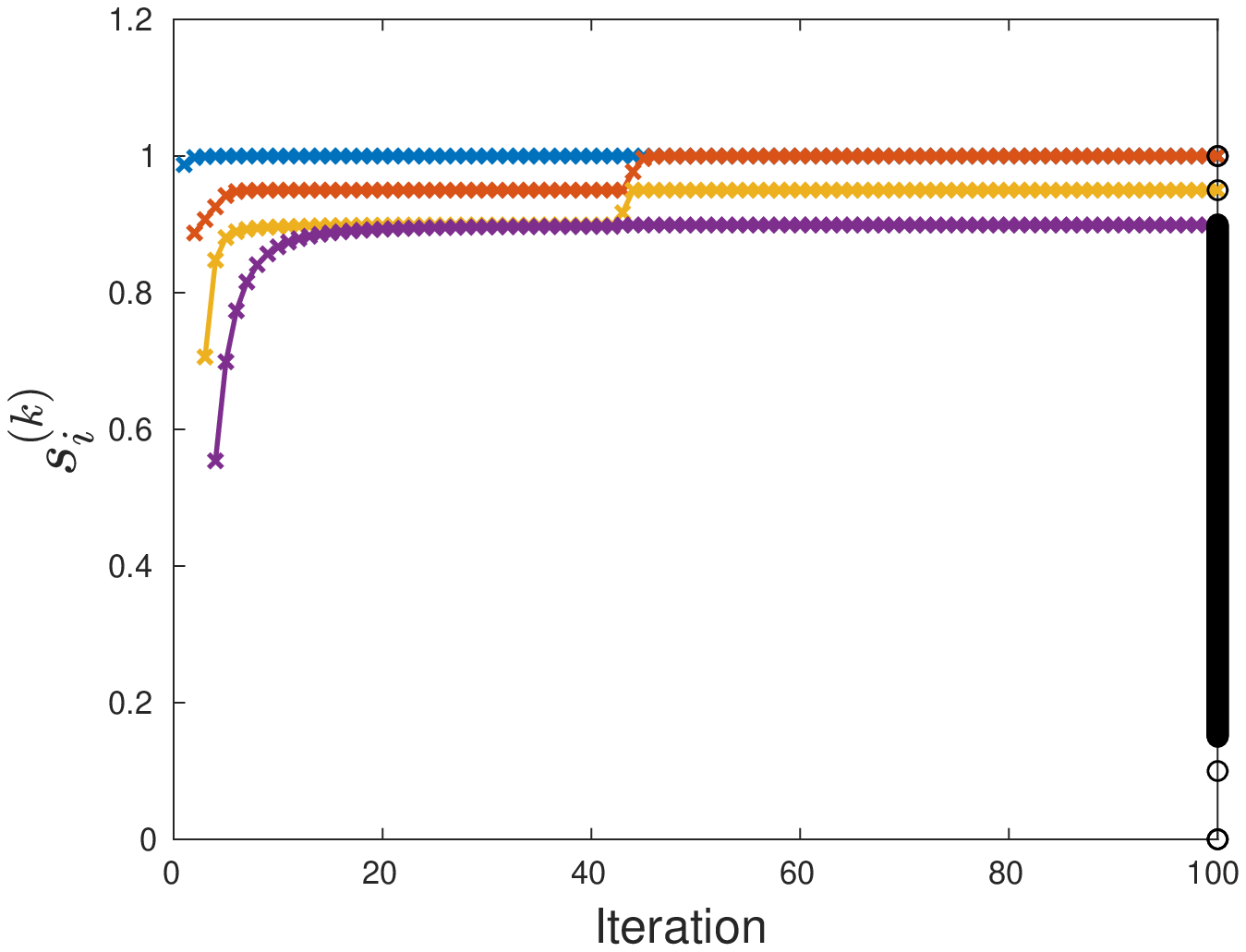}
		\caption{} 
		\label{subfig3a}
	\end{subfigure}
	\begin{subfigure}[t]{0.35\textwidth}
		\centering
		\includegraphics[width=\textwidth]{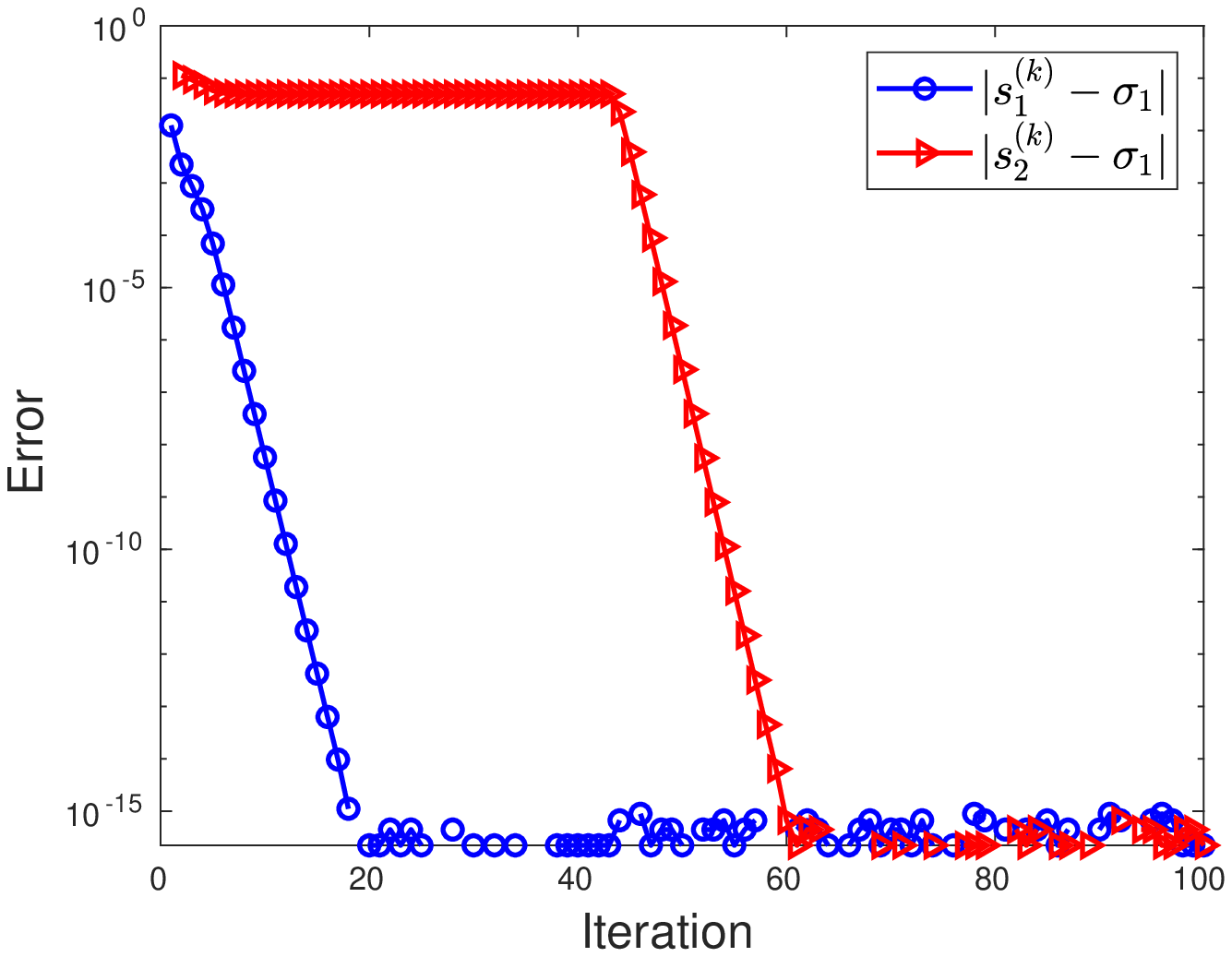}
		\caption{}
	\end{subfigure}
	\begin{subfigure}[t]{0.35\textwidth}
		\centering
		\includegraphics[width=\textwidth]{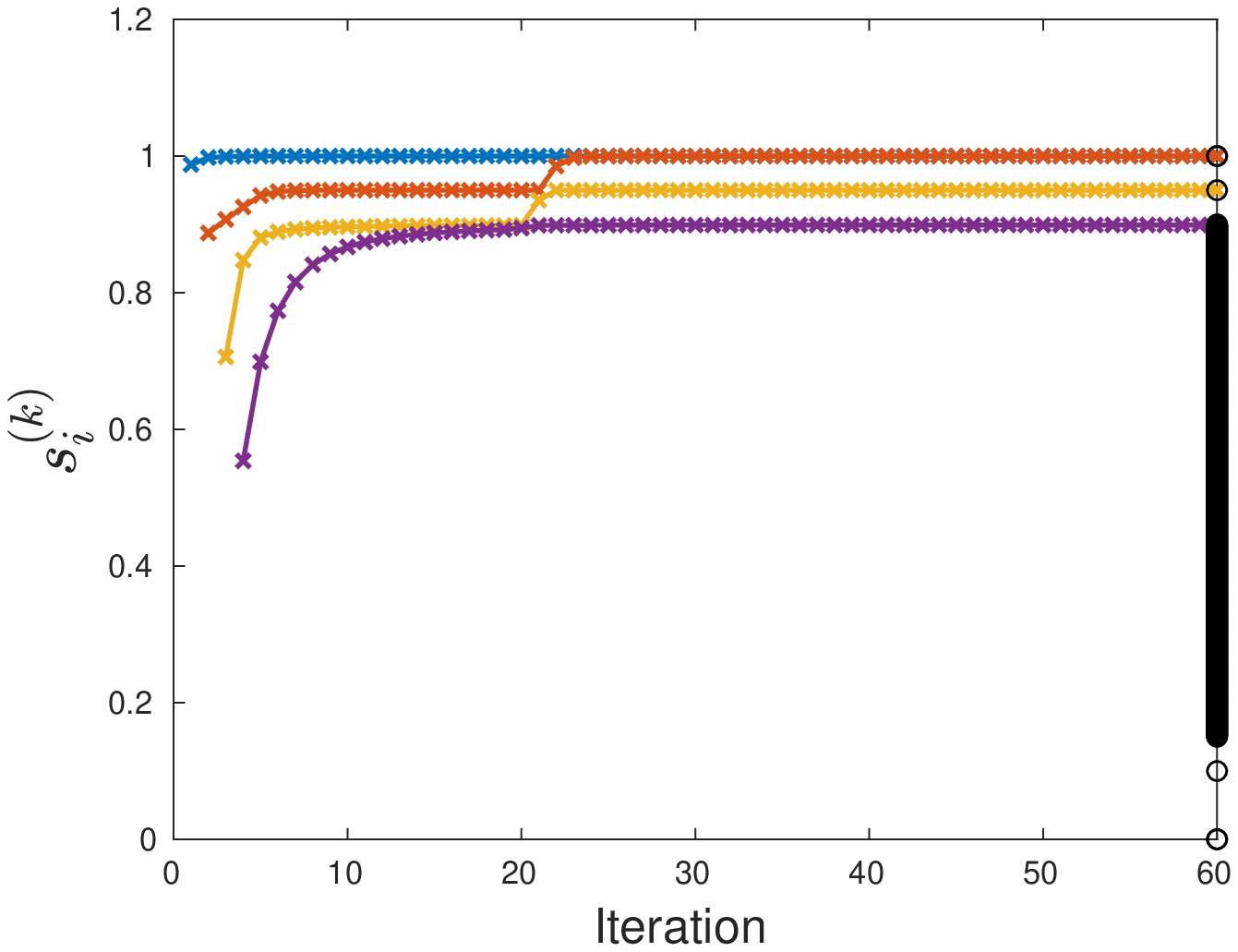}
		\caption{}
		\label{subfig3c}
	\end{subfigure}
	\begin{subfigure}[t]{0.35\textwidth}
		\centering
		\includegraphics[width=\textwidth]{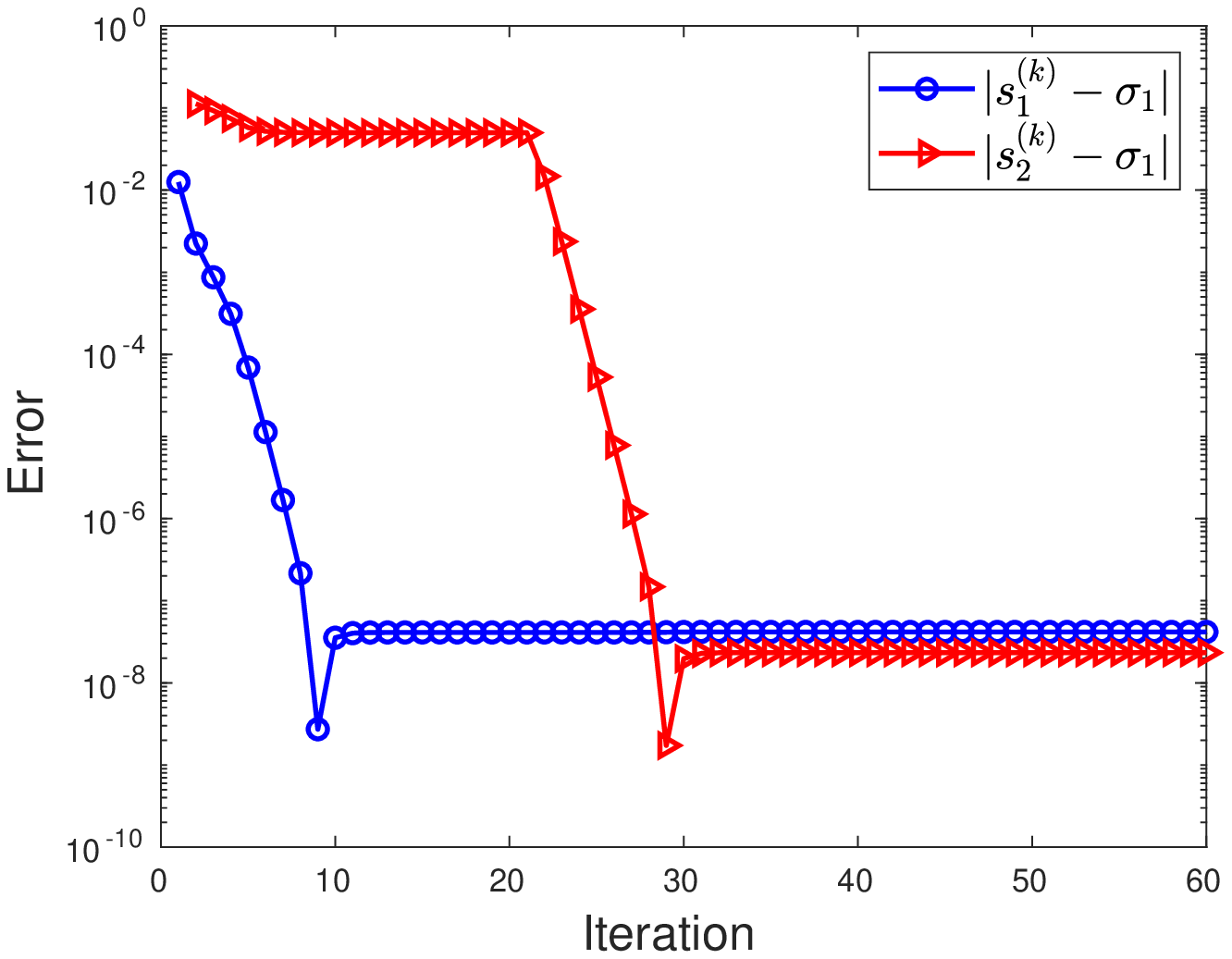}
		\caption{}
		\label{subfig3d}
	\end{subfigure}
	\caption{Convergence and accuracy of approximate largest singular values: (a), (b) full reorthogonalization in double precision arithmetic; (c), (d) full reorthogonalization in single precision arithmetic.}
	\label{fig3}
\end{figure}

Figure \ref{fig3} depicts the convergence history of the first four largest approximations and the error curve corresponding to the 2-multiplicity singular value $\sigma_{1}$. The right vertical lines in subfigures (\subref{subfig3a}) and (\subref{subfig3c}) indicate the values of $\sigma_{i}$ for $i=1,\dots, n$. The experimental results for smallest approximations are similar and we omit them. Subfigures (\subref{subfig3a}) and (\subref{subfig3c}) show that the multiplicities of $\sigma_{1}$ can be determined correctly from the convergence history of $s_{1}^{(k)}$ and $s_{2}^{(k)}$. However, we can also find that $s_{1}^{(k)}$ is not rigorously equal to $s_{2}^{(k)}$ although they are both used to approximate $\sigma_{1}$. This can be observed more obviously in subfigure (\subref{subfig3d}) where the LBRO is performed in single precision arithmetic with roundoff unit $\mathbf{u}=2^{-24}\approx 5.96\times 10^{-8}$.

Table \ref{tab2} shows the errors of approximations to both the 2-multiplicity singular values $\sigma_{1}=\sigma_{2}$ and $\sigma_{n}=\sigma_{n-1}$. For approximations to $\sigma_{1}$ we set $k=100$, while the smallest two Ritz values converge to $\sigma_{n}$ more slowly and we set $k=250$ in this case. We can find that both $s_{1}^{(k)}$ and $s_{2}^{(k)}$ as well as $s_{k}^{(k)}$ and $s_{k-1}^{(k-1)}$ differs by a value of $O(\mathbf{u})$, which is consistent with the upper bound on $\|E\|$. For $\sigma_{n}=10^{-4}$, the relative errors of approximations are much bigger than that of $\sigma_{1}=1.0$ due to a large value of $\sigma_{1}/\sigma_{n}$, which is consistent with \eqref{5.2}.

\begin{table}[htp]
	\centering
	\caption{Accuracy of approximate singular values computed in double and single precision arithmetic.}
	\begin{tabular}{l|lll}
		\toprule
		Work precision & $|s_{1}^{(100)}-\sigma_{1}|/\sigma_{1}$	 & $|s_{2}^{(100)}-\sigma_{2}|/\sigma_{2}$ &$|s_{1}^{(100)}-s_{2}^{(100)}|$ \\   
		\hline
		double  & $2.22\times 10^{-16}$ & $2.22\times 10^{-16}$& $4.44\times 10^{-16}$ 	\\   
		single & $4.17\times 10^{-8}$ & $2.33\times 10^{-8}$ & $1.83\times 10^{-8}$  \\ 
		\hline
		\ & $|s_{250}^{(250)}-\sigma_{n}|/\sigma_{n}$	 & $|s_{249}^{(249)}-\sigma_{n-1}|/\sigma_{n-1}$ &$|s_{250}^{(250)}-s_{249}^{(249)}|$ \\  \hline
		double  & $1.30\times 10^{-12}$ & $1.08\times 10^{-12}$& $2.38\times 10^{-16}$	\\  
		single & $7.93\times 10^{-6}$ & $3.27\times 10^{-5}$ & $2.48\times 10^{-9}$ \\  
		\bottomrule 
	\end{tabular}
	\label{tab2}
\end{table}

\section{Conclusion} \label{sec6}
By writing various types of reorthogonalization strategies in a general form, we have made a backward error analysis of the LBRO in finite precision arithmetic. For the $k$-step LBRO($A,b$), we have shown that: (1). the computed $B_k$ is the exact one generated by the $k$-step LB($A+E,b+\delta_{b}$), where $\|\delta_{b}\|/\|b\|=O(\mathbf{u})$ and $\|E\|/\|A\|$ has the same order of magnitude as the sum of $\mathbf{u}$ and orthogonality levels of $U_{k+1}$ and $V_{k+1}$; (2). if we denote the two orthonormal matrices generated by LB($A+E,b+\delta_{b}$) in exact arithmetic by $\bar{U}_{k+1}$ and $\bar{V}_{k+1}$, respectively, then the first order errors of $\|U_{k+1}-\bar{U}_{k+1}\|$ and $\|V_{k+1}-\bar{V}_{k+1}\|$ are of the same orders of magnitude as orthogonality levels of $U_{k+1}$ and $V_{k+1}$, respectively. Thus the $k$-step LBRO is mixed forward-backward stable as long as the orthogonality of $U_{k+1}$ and $V_{k+1}$ are good enough. This result is then used to show backward stability of the LBRO based SVD computation algorithm and LSQR algorithm. Several numerical experiments are made to confirm the results.

\section*{Acknowledgements}
The authors would like to thank Dr. Long Wang and Professor Weile Jia for their generous supports and stimulating discussions.


\bibliographystyle{elsarticle-harv} 
\bibliography{refs}

%
%
\end{document}